\documentclass[a4paper]{amsart} %singlecoloumn
\usepackage[utf8]{inputenc}
\usepackage[english]{babel}
\usepackage{amsmath}
\usepackage{fontenc}
\usepackage{amsfonts}
\usepackage{amssymb}
\usepackage{amsaddr}
\usepackage{comment}
\usepackage[makeroom]{cancel}
\usepackage{mathtools}
\usepackage{mathrsfs}
\usepackage{stmaryrd}
\usepackage{tikz}
\usepackage{float}
\usepackage{tikz-cd}
\usepackage{tikz}
\makeindex
\usepackage{hyperref}

\renewcommand{\eqref}[1]{(\ref{#1})}

\usepackage{graphicx}

\theoremstyle{plain}
\newtheorem{theorem}{Theorem}[section]
\newtheorem{corollary}[theorem]{Corollary}
\newtheorem{lemma}[theorem]{Lemma}
\newtheorem{proposition}[theorem]{Proposition}

\theoremstyle{definition}
\newtheorem{definition}[theorem]{Definition}

\newtheorem{remark}[theorem]{Remark}

\numberwithin{equation}{section}
\makeindex

\begin{document}
	
	%newcommands declared here
	\newcommand{\mMC}{\text{MC}}
	\newcommand{\md}{\text{d}}
	\newcommand{\comm}[2][red]{\textbf{\textcolor{#1}{#2}}} %commentaries
	\newcommand{\imp}[1]{\textbf{\emph{#1}}} %keywords in definition
	\newcommand{\eqd}{$\coloneqq$ } %defining equal
	\newcommand{\meqd}{\coloneqq} %defining equal in mathmode
	\newcommand{\mad}{\text{ad} }
	\newcommand{\minc}[1]{\underset{#1}{\in}}
	\newcommand{\meqc}[1]{\underset{#1}{=}}
	\newcommand{\mtr}{\text{tr}}
	\newcommand{\abs}[1]{\lvert #1 \rvert}
	\newcommand{\ecomm}[2][green]{\textbf{\textcolor{#1}{#2}}} %commentaries
	\newcommand{\dossier}[2][blue]{{\textcolor{#1}{#2}}} %reference to handwritten notes
	\newcommand{\rnum}[1]{\uppercase\expandafter{\romannumeral #1\relax}} %handmade roman numerals not interfering with page numbering
	\newcommand{\mId}{\text{Id}}
	\newcommand{\mand}{\text{and}}
	\newcommand{\mst}{\text{s.t.}}
	\newcommand{\mmF}{\mathfrak{F}}
	\newcommand{\mmg}{\mathfrak{g}}
	\newcommand{\mmS}{\mathfrak{S}}
	\newcommand{\mmMC}{\mathfrak{MC}_\bullet}
	\newcommand{\mHom}{\text{Hom}}
	\newcommand{\mcC}{\mathcal{C}}
	\newcommand{\mcA}{\mathcal{A}}
	\newcommand{\mcD}{\mathcal{D}}
	\newcommand{\mcO}{\mathcal{O}}
	\newcommand{\mObj}{\text{Obj}}
	\newcommand{\mop}{\text{op}}
	\newcommand{\mSimp}{\text{Simp}}
	\newcommand{\mSet}{\text{Set}}
	\newcommand{\mcoker}{\text{coker}}
	\newcommand{\mker}{\text{ker}}
	\newcommand{\mIm}{\text{Im}}
	\newcommand{\mcone}{\text{cone}}
	\newcommand{\mmod}{\text{mod}~}
	\newcommand{\mwith}{\text{with}}
	\newcommand{\mnew}{\text{new}}
	\newcommand{\mby}{\text{by}}
	\newcommand{\mmin}{\text{min}}
	\newcommand{\ubr}[1]{\underbrace{#1}}
	\newcommand{\nonu}{\nonumber}
	\newcommand{\mMF}{\mmF}
	\newcommand{\minitial}{\text{initial}}
	\newcommand{\mcurv}{\text{curv}}
	\newcommand{\mhK}{\mathbb{K}} 
	\newcommand{\mStub}{\text{Stub}}
	\newcommand{\poly}{\text{poly}}
	\newcommand{\LtR}{{\tilde{L}_{\mathbb{R}}}}
	\newcommand{\LR}{{L_\mathbb{R}}}
	\newcommand{\Lt}{\tilde{L}}
	\newcommand{\mBCH}{\mathrm{BCH}}
	\newcommand{\mfinal}{\mathrm{final}}
	\newcommand{\mCom}{\mathrm{Com}_\infty}
	\newcommand{\kDerCom}{\emph{Der}_{\mCom}} %Denotes the dgLA - the closed degree 0 elements of this dgLA describe Com infty Derivations (see below)
	\newcommand{\DerCom}{\mathrm{Der}_{\mCom}} %Denotes the Com infty Derivations (can be seen as a subalgebra concentrated in degree 0 of the dgLA above - then it forms a Lie algebra)
	\newcommand{\mDef}{\mathrm{Def}}
	\newcommand{\mDefCom}{\mDef_{\mCom}}
	\newcommand{\mP}{\mathcal{P}} %Operadic P in mathmode
	\newcommand{\mFree}{\mathbb{F}} %Free (i.e. free over operad) Algebra
	\newcommand{\mCoFree}{\mathbb{F}^{\mathrm{C}}} %CoFree CoAlgebra
	\newcommand{\mPd}{\mP^{\vee}} %Koszul Dual of Operad P
	\newcommand{\mdg}{\mathrm{dg}}
	\newcommand{\LiA}{$L_\infty$-algebra }
	\newcommand{\LiAn}{$L_\infty$-algebra}
	\newcommand{\sLiA}{$\mmS L_\infty$-algebra }
	\newcommand{\sLiAn}{$\mmS L_\infty$-algebra}
	\newcommand{\Li}{$L_\infty$}
	\newcommand{\sLi}{$\mmS L_\infty$}
	\newcommand{\antishriek}{\text{\raisebox{\depth}{\textexclamdown}}}
	\newcommand{\mBVi}{{BV}_{\infty}}
	\newcommand{\BVi}{$\mBVi$}
	\newcommand{\mHTT}{\mathrm{HTT}}
	\newcommand{\mMor}{\mathrm{Mor}}

	\title{An Obstruction Theory for the Existence of Maurer-Cartan Elements in curved $L_\infty$-algebras and an Application in Intrinsic Formality of $P_\infty$-Algebras.}
	
	\author{Silvan Schwarz}
	\thanks{The author has been partially supported by the ERC starting grant 678156 GRAPHCPX}
	\address{ETH Z\"urich}
	\email{silvan.schwarz@math.ethz.ch}

	\date{\today}% It is always \today, today,
	
	\begin{abstract}
		Let $\mmg$ be a curved $L_\infty$-algebra endowed with a complete filtration $\mmF \mmg$. Suppose there exists an integer $r \in \mathbb{N}_0$ for which the curvature $\mu_0$ satisfies $\mu_0 \in \mmF_{2r+1} \mmg$ and the spectral sequence yields $E_{r+1}^{p,q} =0$ for $p,q$ with $p+q=2$. We prove that then a Maurer-Cartan element exists. In addition, we show, as a typical application, that for $P$ a possibly inhomogeneous Koszul operad with generating set in arities 1,2 (e.g. $P$=Com,As,BV,Lie,Ger), a $P_\infty$-algebra $A$ is intrinsically formal if its twisted deformation complex $\mathrm{Def}(H(A)\stackrel{\mathrm{id}}{\to} H(A))$ is acyclic in total degree 1.
%		Let $\mmg$ be a curved $L_\infty$-algebra endowed with a filtration that is descending, bounded above, complete and compatible with the $L_\infty$-algebra structure.
		
	\end{abstract}	
	\maketitle
	\tableofcontents
	\section{Introduction}
	\label{Motivation}
	%aus W70B (inkl. W70Z) und W71C
	%ueberarbeitete/aktualisierete Version aus W97A
	
	Maurer-Cartan elements play an important role when it comes to finding morphisms which are compatible with some given structures, like the study of $\infty$-morphisms between $P_\infty$-algebras, the study of intrinsic formality of $P_\infty$-algebras or the existence of curved twisting morphisms between co-nilpotent curved co-properads and (not necessarily augmented) dg properads, to name a few.\\
	We show in Theorem \ref{MainThm} that for a curved $L_\infty$-algebra $\mmg$ endowed with a descending bounded above and complete filtration $\mmg= \mmF_1 \mmg \supset \mmF_2 \mmg \supset \ldots$ there is the following implication: If there exists a $r \in \mathbb{N}_0$ for which the curvature satisfies $\mu_0 \in \mmF_{2r+1} \mmg =0$ and the $r+1$st page of the spectral sequence of $\mathfrak{g}$ vanishes $E_{r+1}^{p,q} = 0$ for all $p,q$ with $p+q=2$, a Maurer-Cartan element always exists. Moreover, this Maurer-Cartan element lies in $\mmF_{r+1} \mmg$.\\%siehe auch Kommentar auf W97A2  
	In the second part of this paper we focus on an application of this result in the study of intrinsic formality and formulate a sufficient condition for e.g. $BV_\infty$-algebras\footnote{For the operad $BV_\infty$ as defined in \cite{TonksBV}, Section 1.4.}, or more general $P_\infty$-algebras for $P$ being a possibly \emph{inhomogeneous Koszul operad}\footnote{That is to extend the Definition of being Koszul to operads that do not only involve quadratic but also linear relations, see \cite{TonksBV}, Appendices A, B. Sometimes, this is also refered to as \emph{linear quadratic Koszul operad}} generated in arities 1,2, to be intrinsically formal.\\
	To translate the search of intrinsic formality into the language of Maurer-Cartan elements, we recall the notion of \emph{deformation complex}, an $L_\infty$-algebra\footnote{It is a $\mmS L_\infty$-algebra to be precise, that means there is an additional minor modification on the degrees.} which has the property that its elements correspond to $U(P^\antishriek)$-co-algebra morphisms (where $U:\mathrm{dgVect} \to \mathrm{grVect}$ denotes the forgetful functor) and moreover Maurer-Cartan elements coincide with $\infty$-morphisms, i.e. $U(P^\antishriek)$-co-algebra-morphisms of the co-free $U(P^\antishriek)$-algebras $\mathbb{F}_{U(P^\antishriek)}^c(A)$ and $\mathbb{F}_{U(P^\antishriek)}^c(A)$ that in addition are also compatible with the differentials emerging from the relative bar-constructions.\\
	However, since the study of intrinsic formality focuses on $\infty$-quasi-isomorphisms rather than $\infty$-morphisms, this falls short. Therefore, to establish the existence of an $\infty$-quasi-isomorphism between two $P_\infty$-algebras $A$ and $B$, we start with a quasi-isomorphism of the underlying dg vector spaces $f: (A,d_A) \to (B,d_B)$. This map induces a $U(P^\antishriek)$-co-algebra-morphism $F: \mathbb{F}_{U(P^\antishriek)}^c (U(A)) \to \mathbb{F}_{U(P^\antishriek)}^c (U(B))$ that by construction satisfies the `quasi-isomorphism-condition' but fails to constitute an $\infty$-morphism.\\
	Our approach is to look at the deformation complex twisted by $F$. Twisting an $L_\infty$-algebra means to look at the same underlying graded vector space but with brackets slightly altered. One of the main feature of twisted $L_\infty$-algebras is that $x$ is a Maurer-Cartan element of the $L_\infty$-algebra twisted by $F$ if and only if $x+F$ is a Maurer-Cartan element of the original non-twisted $L_\infty$-algebra. Since twisting an $L_\infty$-algebra by a non-Maurer-Cartan element leads to a curved $L_\infty$-algebra, that is an $L_\infty$-algebra with a non-vanishing 0-bracket, a twisted deformation complex does in general provide us merely with a curved $L_\infty$-algebra. We endow the twisted deformation complex with a descending bounded above and complete filtration compatible with the curved $L_\infty$-algebra structure such that elements of filtration degree $\geq 2$ do not alter the part relevant for the `quasi-isomorphism' condition.
	If we can prove the existence of a Maurer-Cartan element on the twisted deformation complex that carries at least filtration degree 2, this yields an $\infty$-morphism that also is an $\infty$-quasi-isomorphism as it behaves as $F$ with regards to this question.\\
	In this manner, the search of $\infty$-quasi-isomorphism can be conducted by finding Maurer-Cartan elements on curved $L_\infty$-algebras and as such gives a typical application of Theorem \ref{MainThm}.\\
	Moreover, under the assumption of $P$ being a possibly inhomogeneous Koszul operad with the generating set in arities 1,2 the study of intrinsic formality of $P_\infty$-algebras provides a framework in which the choice of $f$ to be the identity-morphism (and hence in particular a quasi-isomorphism of dg vector spaces) satisfies all the requirements of Theorem \ref{MainThm}, as we will explain in Theorem \ref{Main Application Theorem}.\\
	Theorem \ref{MainThm} extends previous work on obstruction theoretical approaches to the existence of Maurer-Cartan elements on differential graded Lie-algebras (see \cite{TonksBV}, Theorem 52) to curved $L_\infty$-algebras.\\
	\subsection{Overview of Results}
	\label{introduction}
	In this paper we deal with curved $L_\infty$-algebras (introduced in Section \ref{preliminaries}). In short, a curved $L_\infty$-algebra differs from a `usual' non-curved $L_\infty$-algebra by also allowing for a non-trivial zero-bracket $\mu_0$. The $L_\infty$-algebra equations it has to satisfy are the same as in the non-curved case but with the sums extended to start at zero (cf. Equation \eqref{LinftyEq}).\\
	For the scope of this paper we do focus exclusively on curved $L_\infty$-algebras which are endowed with a descending, bounded above and complete filtration that is compatible with the curved $L_\infty$-algebra structure. This means for a curved $L_\infty$-algebra $\mmg$ being of such type it must satisfy $\mmg = \mmF_1 \mmg \supset \mmF_2 \mmg \supset \mmF_3 \mmg \supset \ldots$, the degree of filtration must add under $L_\infty$-algebra brackets, that is to say $x_i \in \mmF_{k_i} \mmg$ for $i=1, \ldots, n$ implies $\mu_n ( x_1, \ldots, x_n) \in \mmF_{k_1 + \ldots + k_n} \mmg$, and $\mmg = \lim\limits_{\leftarrow} \mmg /\mmF_p \mmg$ has to hold.\\%Achtung: Dies ist kein colimit, sondern ein inverses Limit -> vgl Draft mit Komentaren aus Meeting vom 26.5.2021
	This additional requirement in turn allows us to speak about the Maurer-Cartan equation (notice that it starts at $n=0$), that is
	\begin{align}
		\label{MC-Eq unsshifted} %Dolg GMT pg. 2
		\sum_{n=0}^{\infty} \frac{1}{n!} \mu \big((x[1])^{\odot n}\big)=0
	\end{align}
	and call elements satisfying this equation Maurer-Cartan elements.\\
	In contrast to the setting of non-curved $L_\infty$-algebras, zero is not a Maurer-Cartan element any more and the natural question comes up if for a given curved $L_\infty$-algebra a Maurer-Cartan element exists.\\
	Although $(\mmg, \mu_1)$ is not a filtered complex (the `differential' fails to square to zero), we can exploit $\mu_1^2 (x) =- \mu_2 (\mu_0,x)$, i.e. we have that the differential squares to zero up to the filtration degree of $\mu_0$. More precisely, we apply the construction of a spectral sequence of a filtered complex on $(\mmg, \mu_1)$ to construct $E_0,E_1, \ldots, E_{r+1}$, which we still call pages of the spec.seq. by abuse of notation, in the usual manner provided the curvature $\mu_0$ lies in $\mmF_{2r+1} \mmg$.\\
	We show that under the sole premise of having a vanishing $r+1$st page of spec.seq. for total degree 2, i.e. $E_{r+1}^{p,q}=0$ for $p+q=2$, and the curvature satisfying $\mu_0 \in \mmF_{2r+1} \mmg$, a curved $L_\infty$-algebra gives rise to a Maurer-Cartan element.
	\begin{theorem}
		\label{MainThm}
		Let $\mmg$ be a curved $L_\infty$-algebra equipped with descending, bounded above and complete filtration $	\mmg= \mmF_1 \mmg \supset \mmF_2 \mmg \supset \mmF_3 \mmg \supset \ldots$	compatible with the curved $L_\infty$-algebra structures.\\
		If there exists an integer $r \in \mathbb{N}_0$ for which the curvature $\mu_0$ of $\mmg$ satisfies $\mu_0 \in \mmF_{2r+1} \mmg$ and the spectral sequence of $\mmg$ vanishes at the $r+1$st page for total degree 2, i.e. $E_{r+1}^{p,q}=0$ for all $p,q$ with $p+q=2$, then there exists a Maurer-Cartan element $\alpha \in \mMC(\mmg)$ that satisfies $\alpha \in \mmF_{r+1} \mmg$.
	\end{theorem}
	Since twisting of a curved \LiA by a Maurer-Cartan element leads to a non-curved $L_\infty$-algebra, Theorem \ref{MainThm} equivalently shows the existence of a non-Abelian non-curved $L_\infty$-algebra structure on $\mmg$ which is compatible with the filtration.\\	
	By applying the aforementioned results to the study of intrinsic formality of $P_\infty$-algebras, we show the following sufficient condition for a $P_\infty$-algebra $A$ to be intrinsically formal
	\begin{theorem}
		\label{Intrinsic Formality}
		Let $A$ be a $P_\infty$-algebra for $P$ a possibly inhomogeneous Koszul operad generated in arities 1,2. If the twisted deformation complex\footnote{See Definition \ref{Deformation Complex} for the definition of the Deformation Complex.} ${\mathrm{Def}(H(A)\stackrel{\mathrm{Id}}{\to} H(A))}$ is acyclic in total degree 1 (i.e. $\mmF_p H^q(\mathrm{Def}(H(A)\stackrel{\mathrm{Id}}{\to} H(A)))=0$ for all $p,q$ with $p+q=1$), then $A$ is intrinsically formal as a $P_\infty$-algebra .
	\end{theorem}
	As the \emph{Batalin-Vilkovisky-operad} $BV$ is of such type, we get the following corollary.
	\begin{corollary}
		Let $A$ be a $BV_\infty$-algebra. If $A$ is subject to $\mmF_p H^q(\mathrm{Def}(H(A)\stackrel{\mathrm{Id}}{\to} H(A)))=0$ for all $p,q$ with $p+q=1$, then $A$ is intrinsically formal as a $BV_\infty$-algebra
	\end{corollary}
	In Section \ref{preliminaries} we introduce some basic notation and recall the definition of spectral sequences. Section \ref{power} is dedicated to the description of (curved) \LiA as cohomological vector fields. In Section \ref{Main} we prove Theorem \ref{MainThm}. Finally, Section \ref{Applications} is used for introducing intrinsic formality and the deformation complex as well as proving Theorem \ref{Intrinsic Formality}. In Appendix \ref{Appendix} we provide proofs to the equivalence of the various descriptions of (curved) $L_\infty$-algebras.
	\section{Preliminaries}
	\label{preliminaries}
	\subsection{Curved $L_\infty$-algebra}
	\begin{definition}[Curved $L_\infty$-algebra]
		\label{Curved Linfty}
		%aus [Markl Diagrams] pg. 100 und [Thomas Draft I] pg. 5u.
		%vgl W114A4 - Ist kein Definitionsproblem bzgl. Def von P_\infty-Derivationen von weiter hinten -> hier haben wir Derivation von Algebren (Assoziativen Algebren?) , z.B. so wie in [Thomas PhD] pg. 16, und somit klarerweise kein Operadisches Differential vorhanden
		A curved $L_\infty$-algebra structure on a graded vector space $\mmg$ is a collection of linear maps
		\begin{align*}
			\mu_n : S^n (\mmg [-1]) \to \mmg[-1]
		\end{align*}
		of degree 1 that satisfy the so-called curved $L_\infty$-algebra relations
		\begin{align}
			\label{LinftyEq}
			\sum_{\substack{i,j\in \mathbb{N}_0\\i+j=n}} \sum_{\sigma\in \mathrm{UnSh}(i,j)} \hspace{-0,4cm} \epsilon(\sigma,x_1, \ldots, x_n) \mu_{j+1} ( \mu_i (x_{\sigma(1)} \odot\ldots \odot x_{\sigma(i)}) \odot x_{\sigma(i+1)} \odot \ldots \odot x_{\sigma (n)})=0,
		\end{align}
		with $\epsilon$ being the shifted Koszul sign (that is the usual Koszul sign w.r.t. elements viewed in the shifted graded vector space) and $\mathrm{UnSh}(i,j)$ denoting the $(i,j)$-unshuffles.
	\end{definition}
	It is immediate from the definition that a zero curvature (i.e. $\mu_0=0$) curved \LiA coincides with a usual $L_\infty$-algebra (cf. \cite{ZwiebachL}, Equation (3.11)). To simplify the notation, we will use commata to separate the arguments of the $n$-brackets. \\
	For the scope of this paper we mean by \LiA only non-curved (a.k.a. flat) $L_\infty$-algebras and always write curved \LiA when we also allow for non-zero curvature. The same of course also holds for the shifted counterparts (introduced below). We generally use the notation $\mmg$ for both the \LiA structure as well as the underlying vector space.\\
	To avoid ubiquitous suspensions we work with shifted (curved) $L_\infty$-algebras instead.
	\begin{definition}[Shifted (curved) $L_\infty$-algebra]
		\label{SLinfty Algebra}
		A (curved) shifted $L_\infty$-algebra structure\index{(curved) $\mmS L_\infty$-algebra} on a graded vector space $\mmg$ is defined as a (curved) $L_\infty$-algebra structure on the suspension $\mmg[-1]$ and we will simply write (curved) $\mmS L_\infty$-algebra for it
	\end{definition}
	Switching to shifted curved $L_\infty$-algebra does not make any significant difference but helps in keeping the notation as simple as possible. A more profound introduction may be found in \cite{DolgushevEnhancement}.\\
	In the theorems, we work with curved $\mmS L_\infty$-algebras endowed with descending, bounded above and complete filtrations compatible with the $\mmS L_\infty$-algebra structures. This means for a $\mmS L_\infty$-algebra $\mmg$ being of such type it must satisfy
	\begin{align}
		\label{W14B3}
		\mmg = \mmF_1 \mmg \supset \mmF_2 \mmg \supset \mmF_3 \mmg \supset \ldots,\\
		\label{W14B4}
		\mmg = \lim\limits_{\leftarrow} \mmg /\mmF_p \mmg
	\end{align}
	and the degree of filtration must add under $\mmS L_\infty$-algebra brackets.\\
	\begin{definition}[Maurer-Cartan elements]
		\label{Maurer-Cartan elements}
		In the category of curved $\mmS L_\infty$-algebras, the Maurer-Cartan elements\index{Maurer-Cartan element} are defined by
		\begin{align}
			\label{MCE}
			\mMC(\mmg) \meqd \{ \alpha \in \mmg^0 : M(\alpha) =0 \},
		\end{align}
		where
		\begin{align}
			\label{MCEq}
			M (x) \meqd \sum_{k=0}^{\infty} \frac{1}{k!} \mu_k (x, \ldots , x ).
		\end{align}
	\end{definition}
	\begin{definition}[Twisted curved $\mmS L_\infty$-algebra]
		\label{Twisted curved SLinfty algebra}
		For a curved \sLiA ${\mmg=(\mmg, \mu_{n \geq 0})}$ and a degree 0 element $\beta \in \mmg^0$, the twisted curved \sLiA ${\mmg^\beta = (\mmg,\mu_{n \geq 0}^\beta)}$ is defined as being the same underlying vector space $\mmg$ but with the curved \sLiA brackets altered to
		\begin{align}
			\label{Twisted}
			\mu_n^\beta (v_1, v_2, \ldots, v_n) \meqd \sum_{k=0}^{\infty} \frac{1}{k!} \mu_{k+n} (\ubr{\beta, \ldots, \beta}_{\text{k-times}}, v_1, v_2 , \ldots, v_n  ).
		\end{align}
	\end{definition}
	In contrast to the non-curved setting, where we were only allowed to twist by Maurer-Cartan elements, in the curved case we may twist by any degree 0 element and obtain a curved \sLiA out of it (see Lemma \ref{Twisting does not destroy} for further details). \\
	Moreover, twisting a curved $\mmS L_\infty$-algebra by a Maurer-Cartan element results in a non-curved $\mmS L_\infty$-algebra. From this point of view it makes sense why in the subcategory of $\mmS L_\infty$-algebras we only allow for twisting by Maurer-Cartan elements.
	%Eigentlich ist engl. Wikipedia der beste Ort fuer eine Kurzeinfuehrung in Spectral Sequence
	%siehe auch um W50A14 herum
	\subsection{Spectral sequences}
	\label{Spectral Sequence}
	\begin{definition}[Spectral Sequence up to page k]
		\label{Def Spectral Sequence}
		%Siehe Korrigiertes v9 Exemplar, fuer wann und warum k-1 stehen muss
		Let $R$ be a ring. By abuse of notation we define a spectral sequence up to page $k$ as a collection of differential bi-graded R-modules $\{E_{r}^{p,q},d_r\}_{k\geq r \geq 0}$, for which $E_{r+1}^{p,q} \cong H^{p,q} (E_r^{\bullet,\bullet},d_r)$ for all $r \in \{0, \ldots, k-1\}$.\\
		In the case of $k=\infty$ this definition coincides with the classical definition of a spectral sequence (cf. \cite{Weibel}, Section 5).
	\end{definition}
	Throughout this paper we use co-homological conventions, hence the differential is assumed to raise the degree by one. By abuse of notation we will still refer to a spectral sequence even if only the spectral pages up to some $r \in \mathbb{N}_0$ are defined. That being said, having an `up to page r spectral sequence` does a priori no longer allow to use spectral sequence arguments and one should consider them as filtered differential graded vector spaces with the additional attribute of the next page being isomorphic to the co-homology of the previous one.\\
	There is a well-known construction for assigning a spectral sequence to a filtered complex\index{Spectral sequence!of filtered Complex} $(V,d)$ (cf. \cite{Weibel}, Section 5.4), by using the fact that the differential is compatible with the filtration and hence satisfies $d (\mmF_{p} V ) \subset \mmF_{p}V$.\\
	In this construction one sets the zeroth page (also called associated graded) to
	\begin{align}
		\label{zeropage}
		E_0^{p,q} \meqd \mmF_{p} V^{p+q} / \mmF_{p+1} V^{p+q}
	\end{align}
	with differential $d_0^{p,q}: E_0^{p,q} \to E_0^{p,q+1}$ being induced by the original differential $d$ and more general defines rth page\index{Spectral sequence!r th page} as
	\begin{align}
		\label{rthpage}
		E_r^{p,q} \meqd \frac{\{z \in \mmF_p V^{p+q} : d (z) \in \mmF_{p+r} V^{p+q+1}\}}{ \{ x \in \mmF_{p+1} V^{p+q} : d(x) \in \mmF_{p+r} V^{p+q+1} \} \oplus \{d(y) \cap \mmF_p V^{q+p} : y \in \mmF_{p-r+1} V^{p+q-1} \}}.
	\end{align}
	%\begin{align}
	%\label{rthpage}
	%E_r^{p,q} \meqd \frac{\{x \in \mmF_p \mmg^{p+q} \vert d (x) \in \mmF_{p+r} \mmg^{p+q+1}\}}{\mmF_{p+1} \mmg^{p+q} \oplus \{d(y) \cap \mmF_p \mmg^{q+p} %\vert y \in \mmF_{p-r+1} \mmg^{p+q-1} \}}
	%\end{align}
	%Here, the notation has to be understood as the quotient of the numerator by its intersection with the denominator. % \comm{schreiben doch besser selbe Form wie e.g Wikipedia - sonst wird es ein bisschen verwirrend wenn wir hier sagen intersection with denominator (das ist die definition im filtered complex Fall) und dann weiterunten im text die Wohldefiniertheit im curved \sLiA Fall untersuchen 
		There, the differential $d_r^{p,q}: E_r^{p,q} \to E_r^{p+r,q-r+1}$ (note the shifts in the codomain) is given by the restriction of the original differential composed with the quotient map.\\
		However, when we try to apply this construction in the setting of a curved $\mmS L_\infty$-algebra $\mmg$ to $(\mmg, \mu_1)$ we face the problem that the one bracket $\mu_1$ does not square to zero but instead yields%vgl. [Hirsh, Milles Curved Koszul...] pg. 1476, dort ander \mmg vorzeichen wie hier (shifting sollte kein Problem sein, denn die sL_\infty Gleichungen (aus welchen dieser Ausdruck kommt) beruecksichtigen den shift ja bereits, s.d. wir am Ende wieder gleichweit sind wie ohne shift)  -> Markl Diagrams pg. 100 sagt aber wie hier
		\begin{align}
			\label{onebracketsquared}
			\mu_1^2 (x) = -\mu_2(\mu_0,x)
		\end{align}
		and hence does not form a chain complex	(or to put it in other words: $-\mu_2(\mu_0,\cdot)$ measures the inability of $\mu_1$ to square to zero).\\
		Nevertheless, the construction intended for filtered complexes with $\mu_1$ in the role of $d$ remains applicable up to a certain (depending on the degree of filtration of the curvature) page, as we now demonstrate.\\
		Let us assume that $\mmg$ is a curved $\mmS L_\infty$-algebra, equipped with a  descending, bounded above and complete filtration compatible with the curved $\mmS L_\infty$-algebra structures. Moreover, let the  curvature satisfy $\mu_0 \in \mmF_{2r+1} \mmg^1$ for some $r$.\\
		We start by noticing that as graded vector the sth page of the spec.seq.\footnote{For now the object formally defined in Equation \eqref{rthpage}.} $E_{s}$ for $s\leq r+1$ is a valid expression. The only non-trivial thing to show is that $\{ \mu_1 (y) \cap \mmF_p \mmg^{p+q} : y \in \mmF_{p-s+1} \mmg^{p+q-1} \} \subset \{z \in \mmF_{p} \mmg^{p+q} : \mu_1 (z) \in \mmF_{p+s} \mmg^{p+q+1} \}$.\\
		So let $x \in \{ \mu_1 (y) \cap \mmF_p \mmg^{p+q} : y \in \mmF_{p-s+1} \mmg^{p+q-1} \}$. By virtue of $\mu_0 \in \mmF_{2r+1} \mmg^1$ and because of $s \leq r+1$ we have %dont forget, \mu_2 is a degree 1 map as well.
		\begin{align}
			\label{welldef}
			\mu_1 (x)= \mu_1^2(y) \meqc{\eqref{W14B4}}- \mu_2(\underbrace{\mu_0}_{\in \mmF_{2r+1} \mmg^1}, \underbrace{x}_{\in \mmF_{p-s+1} \mmg^{p+q-1}}) \in \mmF_{p-s+1+2r+1} \mmg^{p+q+1}  \substack{\subset\\s \leq r+1} \mmF_{p+s} \mmg^{p+q+1},
		\end{align}
		which yields the inclusion.\\
		Let us point out, that it is absolutely crucial for the curvature $\mu_0$ to be subject to the condition $\mu_0 \in \mmF_{2r+1} \mmg^1$. It can directly be checked that for $s=r+1$ Equation \eqref{welldef} may only hold for $\mu_0 \in \mmF_{2r+1} \mmg^1$. As we will later see, this requirement is also needed in the proof of Theorem \ref{MainThm} itself.\\
		%Next, we want to discuss that the construction of Equation \eqref{rthpage} and its respective `differential' (in quotation marks as it has not been proved yet) being the restriction of the original differential to the nominator forms a spectral sequence, indeed.\\
		Next, we check that that the construction of $E_s^{p,q}$ as described in Equation \eqref{rthpage} and its respective differential as the map induced by the 1-bracket $\mu_1$ forms a spectral sequence on the pages $E_0, \ldots E_{r+1}$, indeed\footnote{For $E_{r+1}$ we may set the differential to be the zero map since the presented construction in general is not applicable to define $d_{r+1}$.}.\\ %induced meint hier Restriktion auf den Zaehler (und als Abb auf Quot.-Space betrachtet)
		For that we need to show that the `differential' $d_s$ (induced by the 1-bracket $\mu_1$) is well-defined and satisfies $d_s^2=0$  for $s < r+1$.\\
		It being well-defined follows directly from the fact that 
		\begin{align*}
			\mu_1 (\{ x \in \mmF_{p+1} \mmg^{p+q} : \mu_1(x) \in \mmF_{p+r} \mmg^{p+q+1} \}  ) \subset  \{\mu_1(y) \cap \mmF_{p+r} \mmg^{q+p+1} : y \in \mmF_{p+1} \mmg^{p+q} \}
		\end{align*}
		and that for a $z \in  \{\mu_1(y) \cap \mmF_p \mmg^{q+p} : y \in \mmF_{p-s+1} \mmg^{p+q-1} \}$ its image under $\mu_1$ satisfies  the following two equations
		\begin{align*}
			\mu_1(z) = \mu_1^2 (y)= -\mu_2 (\underbrace{\mu_0}_{\in \mmF_{2r+1} \mmg^1}, \underbrace{y}_{\in \mmF_{p-s+1} \mmg^{p+q-1}}) \in \mmF_{r+1+p-s+1} \mmg^{p+q+1} \substack{\subset \\s < r+1} \mmF_{p+r+1} \mmg^{p+q+1}\\
		\end{align*}
		and
		\begin{align*}
			\mu_1^2(z)= -\mu_2 (\underbrace{\mu_0}_{\in \mmF_{2r+1} \mmg^1}, \underbrace{z}_{\in \mmF_{p} \mmg^{p+q}}) \in \mmF_{p+2r+1} \mmg^{p+q+2} \substack{\subset \\ s< r+1} \mmF_{p+s+s} \mmg^{p+q+2}.
		\end{align*}
		This in turn proves
		\begin{align*}
			\mu_1 ( \{\mu_1(y) \cap \mmF_p \mmg^{q+p} : y \in \mmF_{p-r+1} \mmg^{p+q-1} \}) \subset \{ x \in \mmF_{p+s+1} \mmg^{p+q+1} : \mu_1(x) \in \mmF_{p+s+s} \mmg^{p+q+2} \}.
		\end{align*}
		It remains to show $d_s^2=0$, that is to say (written out to make the `filtration degree shift' explicit)
		\begin{align*}
			d_s^{p+s,q-s+1} ~ d_s^{p,q} : E_s^{p,q} \to E_s^{p+2s, q-2s+2}
		\end{align*}
		vanishes.\\
		For $[x]\in E_s^{p,q}$ we have $x\in \mmF_p \mmg^{p+q}$ and $\mu_1 (x) \in \mmF_{p+s}\mmg^{p+q+1}$. Applying the differential two, respectively three times yields that
		\begin{align*}
			\mu_1^2 (x) =- \mu_2 (\mu_0,x) \in \mmF_{2r+1 + p} \mmg^{p+q+2} \substack{\subset \\ s < r+1} \mmF_{p+2s+1} \mmg^{p+q+2}\\
			\mu_1^3 (x)= \mu_1^2 (\underbrace{\mu_1(x)}_{\in \mmF_{p+s}\mmg^{p+q+1}}) = -\mu_2 ( \mu_0, \mu_1 (x)) \in \mmF_{2r+1 + p+s} \mmg^{p+q+3}  \substack{\subset \\ s < r+1} \mmF_{p+3s+1} \mmg^{p+q+3},
		\end{align*}
		from which it clearly follows that $\mu_1^2 (x)$ vanishes under the projection to the quotient space $E_s^{p+2s, q-2s+2}$.\\
		The proof of $E_{n+1}^{p,q} \cong H^{p,q}(E_n^{\bullet, \bullet},d_r)$ for $n\leq r$ follows the usual path of the construction of a spectral sequence of filtered complexes (cf. \cite{Weibel}, Sect.5.4).\\
		\subsection{Derivations and Differentials on operadic algebras}
		%Hauptsaechlich aus W110A (+wenig aus W107D,W107A und W107B). Literatur zu diesem Thema: [Getzler,Jones] pg. 22
		When it comes to (co-)derivations and (co-)differentials of (co-)operadic (co-)algebras we mainly follow the definitions of \cite{GetzlerJones}, Sect.2.2, but with co-homological conventions rather than homological. Moreover, we use the notation of \cite{FreBuch}, Vol.\rnum{1}, Sect.0.13, that is in case of $\mathcal{C}$ being an enriched category over $\mathcal{D}$ we distinguish between morphism sets $\mMor_{\mathcal{C}}$ and $\mHom_{\mathcal{C}}$ which denotes hom-objects with values in $\mathcal{D}$.\\
		Notice, that for closed symmetric monoidal categories $M$ (as e.g. $M=grVect$) there is the internal hom-adjunction
		\begin{align*}
			\mMor_M (A \otimes B, C) \simeq \mMor (A, \mHom_M (B,C))
		\end{align*}
		for all $A,B,C \in \mathrm{Obj}(M)$.\\
		Therefore, for $A,B$ being two graded vector spaces, $f\in \mMor_{grVect} (A,B)$ is a linear map that satisfies $f(A_n) \subset B_n$ for all $n\in \mathbb{Z}$. On the other hand (cf. \cite{FreBuch}, Vol.\rnum{1}, Sect.4.4), $g\in \mHom_{grVect} (A,B)$ is only required to be a linear map and hence in general is allowed to change the degree. By definition, $\mHom_{grVect}(A,B)$ carries a graded vector space structure given by $\mHom_{grVect} (A,B)^{(k)} = \{f: A \to B~ \text{linear} : f(A_n)\subset B_{n+k} \forall n\in \mathbb{Z}\}$.\\
		Let $U: dgVect \to grVect$ be the forgetful functor and let $P$ be an operad in $dgVect$ with differential $d_p$.\\
		In the following, we will introduce the three distinctive definitions of a differential on a $U(P)$-algebra, $P$-derivation on a $P$-algebra and differential-capable $P$-derivation.
		\begin{definition}[Differential on $U(P)$-algebra]
			\label{Differential on U(P)}
			Let $A$ be a $U(P)$-algebra (hence, $A$ as well the structure maps $\gamma: \mathbb{F}_{U(P) } (A) \to A$ lie in $grVect$).\\
			We say a graded vector space homomorphism $\xi: A \to A$ of degree 1 is a differential of $A$ if it squares to zero and makes the following diagram commutative:
			\begin{equation}
				\label{Differential on A}
				\begin{gathered}
				%\centering
				\tikzstyle{mynode}= [circle,draw=black!50,]
				\begin{tikzpicture}	
					\node [black] (v1) at (-3,2) {$P (A) $};
					\node [black] (v2) at (1,2) {$P (A) $};
					\node [black] (v3) at (-3,0) {$A$};
					\node [black] (v4) at (1,0) {$A$};
					\node[black,scale=2] (commutative) at (-1,1) {$\circlearrowright$};
					\draw[->] (v1) edge node[midway, above] {$d_P \circ Id_A + Id_P\circ^\prime \xi$} (v2)  ;
					\draw [->] (v1) edge node [midway,right] {$\gamma$} (v3);
					\draw [->] (v3) edge node[midway, above] {$\xi$} (v4);
					\draw [->] (v2) edge node[midway, right] {$\gamma$} (v4);
				\end{tikzpicture},
%				\caption{Differential on $U(P)$-algebra}
%				\label{Differential on A}
				\end{gathered}
			\end{equation}
			where $f\circ g$ acts on $(\mu; a_1, \ldots, a_k)$ by $f \otimes_{\mathbb{S}_n} (g \otimes \ldots \otimes g)$ and $f \circ^\prime g$ by $\sum_{j=1}^k f \otimes_{\mathbb{S}_n} (Id_A \otimes \ldots \otimes Id_A \otimes \ubr{g}_{\text{jth position}} \otimes Id_A \ldots \otimes Id_A)$. \\
			In particular, this requirement is equivalent to $(A,\xi)$ forming a $P$-algebra.
		\end{definition} 
		Let us point out, that for a differential graded vector space $(C,d_C)$, the map $d_{\mathbb{F}_P (C)} := d_P \circ Id_C + Id_P \otimes^\prime d_C$ forms a differential on $\mathbb{F}_{U(P)} (C)$. The so constructed $P$-algebra is the free $P$-algebra $\mathbb{F}_P (C)$.
		\begin{definition}[Derivation of a $P$-algebra]
			Let $A$ be a $P$-algebra or a $U(P)$ algebra.\\
			We call a linear map $\tau: A \to A$ of (degree k) a derivation of $A$ (of degree $k$) or equivalently $P$-derivation (of degree $k$) when it makes the following diagram of graded vector space homomorphisms commutative:
			\begin{equation}
				\label{Derivation on A}
				\begin{gathered}
				%\centering
				\tikzstyle{mynode}= [circle,draw=black!50,]
				\begin{tikzpicture}	
					\node [black] (v1) at (-3,2) {$P (A) $};
					\node [black] (v2) at (1,2) {$P (A) $};
					\node [black] (v3) at (-3,0) {$A$};
					\node [black] (v4) at (1,0) {$A$};
					\node[black,scale=2] (commutative) at (-1,1) {$\circlearrowright$};
					\draw[->] (v1) edge node[midway, above] {$Id_P\circ^\prime \tau$} (v2)  ;
					\draw [->] (v1) edge node [midway,right] {$\gamma$} (v3);
					\draw [->] (v3) edge node[midway, above] {$\tau$} (v4);
					\draw [->] (v2) edge node[midway, right] {$\gamma$} (v4);
				\end{tikzpicture}.
				%\caption{Derivation on $P$-algebra}
				%\label{Derivation on A}
			\end{gathered}
			\end{equation}
		\end{definition}
		Let us emphasise that in contrast to a differential on the $U(P)$-algebra $A$, in the definition of derivation on $A$ there is neither a requirement to square to zero nor to have degree +1. Moreover, the notion of derivation makes sense both in the category of graded vector spaces as well as differential graded vector spaces (also, in the latter case there is no requirement for the derivation to commute with the vector space differential $d_A$).\\
		It is immediate that if $P$ is an operad with zero-differential, the two definitions of differential on a $U(P)$-algebra and derivation on a $P$-algebra do coincide.
		\begin{definition}[Differential-capable $P$-Derivation]
			\label{P-Differential}
			%Selbsterfundener Name, da in Literatur haufig ein Chaos durch verwechselfraudige Notation
			Let $A$ be a $P$-algebra with differential $d_A$. Then we say a $P$-derivation $\beta$ on $A$ is differential-capable relative to $d_A$, if it is of degree $+1$ and satisfies $(d_A+ \beta)^2=0$. 
		\end{definition}
		Because of the requirement $(d_A + \beta)^2 =0$, it is immediate that the sum of the differential on $A$ and a differential-capable $P$-derivation $\beta$ forms a differential on $A$ again.\\
		Later in this paper, we will often work with quasi-free\footnote{Some authors (see \cite{GetzlerJones}) also call this almost-free.} algebras. Recall, that a $P$-algebra $B$ is quasi-free, if $B=\mathbb{F}_{U(P)} (V)$ as underlying graded vector space for some  dg vector space $(V,d_V)$ but we allow $B$ to have a differential other than $d_{\mathbb{F}_{P}(V)}$. More precisely, we allow $B$ to have differential $d_B=d_{\mathbb{F}_{P}(V)} + \nu$ for $\nu$ being a differential-capable $P$-derivation on $\mathbb{F}_{U(P)}(V)$. \\
		However, for several reasons (just think about e.g. the definition of an $\infty$-quasi-isomorphism) we do not want to have the differential (seen as a dg vector space differential) on the $V$-part of $B$ being changed. To this end, we impose on $P$-derivations $\nu$ of $A$ (and consequently on differential-capable $P$-derivations) the additional condition that the composition
		\begin{align}
			\label{W114A6}
			V \stackrel{\cong}{\to} I(V) \stackrel{\eta \circ Id_V}{\to} P(V) \stackrel{\nu}{\to} P(V) \stackrel{\epsilon}{\to} I(V) \stackrel{\cong}{\to} V
		\end{align}
		vanishes, where $\eta$ and $\epsilon$ denote the unit and the augmentation of the operad $P$.\\
		While some authors (\cite{GetzlerJones}) use $\mathrm{Der}_\star$ and $\mathrm{Diff}_\star$ to distinguish $P$-derivations and differential-capable $P$-derivations subject to Equation \eqref{W114A6} from those without, we dispense the use of an additional $\star$, since we always use the ones where the additional requirement is invoked.\\
		Of course, all of the aforementioned definitions \emph{mutatis mutandis} also hold for their dual counterparts.\\
		We further elaborate on this topic in Section \ref{Revisited}, when we introduce the notion of derivation w.r.t. a $U(P)$-algebra morphism. Moreover, by making use of extensions of scalars, we point out how freeness can be used in describing derivations of free $U(P)$-algebras.		
		\section{$\mmS L_\infty$-equation as a power Series }
		\label{power}
		%Siehe Ausdruck W74Z und Meeting 26.5.2021 zur Thematik ob nur auf filtered L_\infty algebra equivalente Formulierung mit Power Serie möglich
		%Siehe W44A, und Anhang von W45A; in W44A wird auch der Umgang mit L_\infty-Homomorphismen beschrieben
		In this section we introduce an equivalent description of (curved) $\mmS L_\infty$-algebras in terms of power series since we will later use this language in the proof of Theorem \ref{MainThm}. One advantage of this alternative definition is that one can circumvent the ubiquitous signs in the curved $(\mmS) L_\infty$-algebra Equation \eqref{LinftyEq}.\\
		Let $\mmg= (\mmg, \{\mu_n\}_{n \geq 0})$ be a curved \sLiAn.
		Following (\cite{MappingSpaces}, Section 4 and \cite{RatHomMap}, Section 1) we use the construction of the power series
		\begin{align*}
			M(x) = \sum_{k=0}^{\infty} \frac{1}{k!} \mu_k (x,\ldots,x)
		\end{align*}	
		as starting point.\\
		So far nothing new, we already knew that e.g. solutions of the equation $M(x)=0$ are some special elements, called the Maurer-Cartan elements.\\
		Now comes the trick: We can also do the formal construction of $M(x)$ for some general n-ary degree 1 operations $\{\mu_n\}_{n\geq 0}$ which do not satisfy the $\mmS L_\infty$-algebra equations. As it turns out, the question of these operations $\mu_n$ forming $\mmS L_\infty$ brackets directly translates to whether the power-series $M(x)$ itself is a solution to a way simpler looking equation. Let us make the statement more concrete.\\
		Let $\mmg = (\mmg, \{\mu_n\}_{n \geq 0})$ be a graded vector space equipped with n-linear graded vector space homomorphism $\mu_n$ of degree 1 for all $n \geq 0$, which we do refer to as n-brackets. Let $R$ be a nilpotent graded ring. By extension of scalars we may extend the n-brackets $\{\mu_n\}_{n \geq 0}$ to the completed tensor product $\mmg \hat{\otimes} R$ and apply the construction of Equation \eqref{MCEq} to it, i.e. we have
		%Vgl. [Dolg. Rogers; GMT] pg. 8 Eq (2.22), sowie [LV] pg. 24
		\begin{align}
			\label{MR}
			M^R (x \hat{\otimes} r) = \sum_{n=0}^{\infty} \frac{1}{n!} \pm \mu_n (x, \ldots, x) \hat{\otimes} r \cdot \ldots \cdot r,
		\end{align}
		where the sign is due to the Koszul sign convention.\\
		It turns out (see Appendix \ref{Appendix} for more details) that for the n-brackets $\{\mu_n\}_{n \geq 0}$ to satisfy the curved $\mmS L_\infty$-algebra equation it is tantamount to
		\begin{align}
			\label{powereq}
			D M^R (x) [M^R (x)]=0 \qquad \forall x \in (\mmg\hat{\otimes} R)^0,
		\end{align}
		for all graded nilpotent Rings $R$ and for all degree 0 elements in $\mmg\hat{\otimes} R$ with $D$ denoting the differential of the multilinear map and $[\ldots]$ being the point of evaluation.\\
		This alternative definition is exactly the description of ($\mmS$)$L_\infty$-algebras as co-homological vector fields which is well established in physics (cf. \cite{ZwiebachL}, Sections 2.1 and 2.2).\\%Zwiebach pg. 6- pg. 9; m.E. ist der Punkt auch, dass L_\infinity Algebra Struktur auch als Square Zero Coderivation (S(V) ist ja eine Coalgebra) Q definiert werden kann. Eine Coderivation ist vollstandig durch lineare Abb F:S(V) -> V definiert. Laut [B.F., T.W.,V.T.; The Rational Homotopy of Mapping Spaces of E_n Operads] pg.19f. ist aber eine solche Abbildung vollstaendig durch power series beschrieben.
		%vgl  [B.F., T.W.,V.T.; The Rational Homotopy of Mapping Spaces of E_n Operads] pg.20 o. fuer die Beschreibung des Rings R
		If the $\mmS L_\infty$-algebra can be fully described by $\mmg = (\mmg, M^R)$, one may ask how to obtain an explicit expression for $\mu_n(x_1, \ldots, x_n)$ for some homogeneous elements $x_1, \ldots, x_n \in \mmg$ out of this (in Equation \eqref{MR} we only allowed for all the elements in the arguments being the same). The solution to this problems is obtained by graded polarisation and works as follows: Consider $\epsilon_i$ to be a formal variables of degree $-\mathrm{deg}(x_i)$ and let $R$ be the graded Ring $R\meqd \mathbb{K}[\epsilon_1, \ldots, \epsilon_n]/(\epsilon_1^2, \ldots, \epsilon_n^2)$. Then, $\mu_n (x_1, \ldots, x_n)$ is given by the $\epsilon_1, \ldots, \epsilon_n$ coefficient of the $\epsilon_1 \cdots \epsilon_n$ monomial of $M^R (x_1 \otimes \epsilon_1, \ldots, x_n \otimes \epsilon_n)$.\\
		%Fuer ein Rechenbeispiel siehe W44A9
		%Statement aus W45A5
		There is yet another equivalent form (see Appendix \ref{Appendix} for more details) of the $\mmS L_\infty$-algebra equation
		\begin{align}
			\label{powereqalt}
			M^{R \otimes R^\prime} (x\hat{\otimes} 1+M^R(x)\hat{\otimes}\epsilon ) = M^R(x)\hat{\otimes}1 \qquad \forall x \in (\mmg \hat{\otimes} R)^0,
		\end{align}
		which we obtain by choosing $\epsilon$ to be a formal variable of degree $-1$ and $R^\prime$ to be the graded algebra $R^\prime= \mathbb{K} [\epsilon]/\epsilon^2$. This last version turns out to be useful when working with twisted curved $\mmS L_\infty$-algebra, as it dramatically simplifies the proof for the twisted curved $\mmS L_\infty$-algebra to be a curved $\mmS L_\infty$-algebra as well.

		\section{Main Theorem}
		\label{Main}
		In this section we focus  to the proof of Theorem \ref{MainThm}.\\
		Let us recall that we work in the setting of curved $\mmS L_\infty$-algebras endowed with a descending, bounded above and complete filtration that is compatible with the curved $\mmS L_\infty$-algebra structure. In particular, we always demand the filtration to start with $\mmg = \mmF_1 \mmg$ (also see Equations \eqref{W14B3} and \eqref{W14B4}), which together with completeness ensures convergence of the infinite sums as e.g. appearing in the Maurer-Cartan equations.
		By switching between the three equivalent descriptions of curved $\mmS L_\infty$-algebra the main theorem can be proved almost in absence of any bigger calculation.\\	
		The following lemma shows, that in the curved $\mmS L_\infty$-algebra setting we are not only allowed to twist by Maurer-Cartan elements but rather by any degree 0 element and still obtain a $\mmS L_\infty$-algebra endowed with a `good' filtration.
		\begin{lemma}
			\label{Twisting does not destroy}
			%Aus W45A13f.
			Let $\mmg=(\mmg, \{\mu_n\}_{n \geq 0})$ be a curved $\mmS L_\infty$-algebra endowed with a descending bounded above and complete filtration $\mmg = \mmF_1 \mmg \supset \mmF_2 \mmg \supset \ldots$ compatible with the $\mmS L_\infty$-algebra structure.\\
			Let $\beta \in \mmg^0$ be an arbitrary degree zero element.\\
			Then $\mmg^\beta = (\mmg,\{\mu_n^{\beta} \}_{n \geq 0})$ with $\{\mu_n^{\beta} \}_{n \geq 0}$ as defined in Equation \eqref{Twisted}, also describes a curved $\mmS L_\infty$-algebra. Moreover, the filtration maintains its qualities of being descending, bounded above, complete and compatible with the curved $\mmS L_\infty$-algebra structure.
		\end{lemma}
		From a geometrical point of view we can interpret a curved $\mmS L_\infty$-algebra $\mmg$ as a formal vector field $Q(x)$ (which, when expanded as a power series, has $\mu_n$ as nth Taylor coefficient) that squares to zero. In this language twisting by $\beta$ corresponds to a displacement of the origin by $-\beta$. Clearly, this new vector field still squares to zero, hence still describes a curved $\mmS L_\infty$-algebra equation. For non-curved $\mmS L_\infty$-algebras however,there is a problem. A vanishing curvature translates to the vector field $Q(x)$ also having a root at the origin. When displacing the origin by some arbitrary $-\beta \in \mmg^0$. this condition in general may no longer hold, hence making the resulting $\mmS L_\infty$-algebra to be curved.\\
		Since this papers focus lies on algebra, let us elaborate on a proof of this claim in the language of algebra.	
		\begin{proof}[Proof of Lemma \ref{Twisting does not destroy}]
			\label{Proof of Twisting does not destroy}
			Let us denote the construction of Equation \eqref{MR} by $M_{\mmg}^R$ and $M_{\mmg^\beta}^R$, respectively, to emphasise which brackets were used.\\
			%Aus W38A1
			Let $x \in (\mmg \hat{\otimes} R)^0$ and $\beta = \beta \hat{\otimes} 1$ (by abuse of notation) for $R$ as in Section \ref{power}.\\
			Our first step is to verify
			\begin{align}
				\label{TwistedAlt}
				M_{\mmg^\beta}^R(x) = M_{\mmg}^R(x+ \beta).
			\end{align}
			On one hand we have
			\begin{align*}
				M_\mmg^R (x + \beta) &= \sum_{n \geq 0} \frac{1}{n!} \mu_n (x + \beta, \ldots, x + \beta)\\
				&= \sum_{k \geq  0} \sum_{j \geq 0} \frac{(j+k)!}{j!k!} \frac{1}{(j+k)!} \mu_{j+k} (\ubr{\beta, \ldots,\beta}_{j~\text{times}}, \ubr{x, \ldots, x}_{k~\text{times}})\\
				&= \sum_{k \geq  0} \sum_{j \geq 0} \frac{1}{k!j!} \mu_{j+k} (\ubr{\beta, \ldots,\beta}_{j~\text{times}}, \ubr{x, \ldots, x}_{k~\text{times}}),
			\end{align*}
			where we used multilinearity, $\abs{x}=0$ and $\abs{\beta}=0$, relabelled the sums $n=j+k$ and made use of some elementary combinatorics, which yields that there are exactly $\binom{j+k}{j}= \frac{(j+k)!}{j!k!}$ combinations to have an $j+k$ bracket with $j$-many $\beta$ and $k$-many $x$.\\
			On the other hand also		
			\begin{align*}
				M_{\mmg^\beta}^R (x) \meqc{\eqref{MR}}  \sum_{n \geq 0} \frac{1}{n!} \mu_n^\beta (x, \ldots, x) \meqc{\eqref{Twisted}} \sum_{n \geq 0} \sum_{j \geq 0} \frac{1}{n! j!} \mu_{n+j}(\ubr{\beta, \ldots,\beta}_{j ~\text{times}},\ubr{x, \ldots, x}_{m~\text{times}})
			\end{align*}
			holds, which proves Equation \eqref{TwistedAlt}.\\
			%folgendes nun aus W45A14
			In turn, Equation \eqref{TwistedAlt} allows for the following relation (with $R^\prime$ as in Section \ref{power})
			\begin{align*}
				M_{\mmg^\beta}^{R\otimes R^\prime}(x \hat{\otimes} 1 +  M_{\mmg^\beta} (x) \hat{\otimes} \epsilon )\meqc{\eqref{TwistedAlt}} M_\mmg^{R\otimes R^\prime} (x\hat{\otimes} 1 + \beta \hat{\otimes} 1+ M_\mmg^R (x+\beta)\hat{\otimes}  \epsilon)\\
				\meqc{\tilde{x} = x + \beta} M_\mmg^{R \otimes R^\prime} ( \tilde{x}\hat{\otimes} 1 + M_\mmg^R (\tilde{x})\hat{\otimes} \epsilon) 	\meqc{\eqref{powereqalt}} M_\mmg^R (\tilde{x})\hat{\otimes} 1 = M_\mmg^R (x+\beta) \hat{\otimes} 1 \meqc{\eqref{TwistedAlt}} M_{\mmg^\beta}^R (x)\hat{\otimes}1,
			\end{align*}
			which is just the $\mmS L_\infty$-algebra equation for $\mmg^\beta$ in the language of Equation \eqref{powereqalt}.
		\end{proof}
		Another technical lemma we need in the proof of the main theorem is the following:
		\begin{lemma}
			\label{Twisting and Spectral}
			Let $\mmg$ be  curved $\mmS L_\infty$-algebra equipped with a descending bounded above and complete filtration $\mmg = \mmF_1 \mmg \supset \mmF_2 \mmg \supset \ldots$ compatible with the $\mmS L_\infty$-algebra structure and let $r \in \mathbb{N}_0$ be an integer for which the curvature $\mu_0$ of $\mmg$ satisfies $\mu_0 \in \mmF_{2r+1} \mmg^1$ and the $r+1$st page of the spec.seq. vanishes, i.e. $E_{r+1}^{\bullet, \bullet}(\mmg)=0$.\\
			If $\alpha \in \mmF_{r+1} \mmg^0$, then twisting of $\mmg$ by $\alpha$ does preserve the condition of having a vanishing $r+1$st page of the spec.seq., i.e. $E_{r+1}^{\bullet, \bullet}(\mmg^\alpha) =0$ holds.\\
			This result also holds for each of the total degrees individually.
		\end{lemma}
		\begin{proof}[Proof of Lemma \ref{Twisting and Spectral}]
			\label{Proof of Twisting and Spectral}
			Recalling Equation \eqref{rthpage}, $E_{r+1}^{p,q} =0$ means that
			\begin{align}
				\label{Spectral Alt}
				\left\{
				\begin{aligned}
					\forall~x \in \mmF_p \mmg^{p+q}~\text{s.t.}~ \mu_1(x)\in \mmF_{p+r+1} \mmg^{p+q+1}\\
					\exists y \in \mmF_{p-r} \mmg^{p+q-1}~\text{s.t.} \\
					\mu_1 (y) \in \mmF_p \mmg^{p+q}~\mand~x-\mu_1 (y) \in \mmF_{p+1} \mmg^{p+q}.
				\end{aligned}
				\right.
			\end{align}
			Let $x \in \mmF_p \mmg^{p+q}$ and $\alpha \in \mmF_{r+1} \mmg^0$, then 
			\begin{align*}
				\mu_1^\alpha (x)= \mu_1 (x)+ \ubr{\mu_2 (\ubr{\alpha}_{\in \mmF_{r+1} \mmg^0}, \ubr{x}_{\mmF_{p} \mmg^{p+q}})}_{\in \mmF_{r+1+p} \mmg^{p+q+1}} + \frac{1}{2!} \mu_3 (\alpha, \alpha,x)+ \ldots
				= \mu_1 (x) +\mathcal{O}( \mmF_{r+1+p} \mmg^{p+q+1})
			\end{align*}
			holds.\\
			But on the other hand, for $y \in \mmF_{p-r} \mmg^{p+q-1}$ we also have
			\begin{align*}
				\mu_1^\alpha (y) = \mu_1 (y) + \ubr{\mu_2 (\ubr{\alpha}_{\in \mmF_{r+1} \mmg^0},\ubr{y}_{\in \mmF_{p-r} \mmg^{p+q-1}}	)}_{\in \mmF_{p+1} \mmg^{p+q}} + \frac{1}{2!} \mu_3( \alpha, \alpha , y) + \ldots = \mu_1 (y) + \mathcal{O} (\mmF_{p+1} \mmg^{p+q}),
			\end{align*}
			allowing us to replace all the expressions in Equation \ref{Spectral Alt} with its twisted counterparts (and vice versa), so particularly the statement $E_{r+1}^{\bullet, \bullet} =0$ does not change under twisting.% Klar, der andere Ausdruck in \eqref{rthpage} ist lediglich der durchschnitt des Zaehlers mit $\mmF_{p+1}$ - siehe auch W50A14
		\end{proof}
		One more technical lemma allows us in the setting of Theorem \ref{MainThm} to twist the curved \sLiA in such a way that the filtration degree of the curvature of the twisted $\mmS L_\infty$-algebras curvature is raised by one when compared to the curvature of the original untwisted $\mmS L_\infty$-algebra.
		
		\begin{lemma}
			\label{Twisting towards the final destination}
			Let $\mmg$ be a curved $\mmS L_\infty$-algebra equipped with a descending, bounded above and complete filtration $ \mmg= \mmF_1 \mmg \supset \mmF_2 \mmg \supset \mmF_3 \mmg \supset \ldots $ compatible with the curved $\mmS L_\infty$-algebra structures.\\
			Let $r \in \mathbb{N}_0$ and $k \geq 2r+1$ be some integer for which the curvature $\mu_0$ of $\mmg$ satisfies $\mu_0 \in \mmF_k \mmg$ and $E_{r+1}^{p,q} =0$ holds for all $p,q$ with $p+q=1$.\\
			Then there exists an $\alpha \in \mmF_{k-r} \mmg^0$ such that $\mu_0 - \mu_1 (\alpha) \in \mmF_{k+1} \mmg^1$ and the curvature $\mu_0^{-\alpha}$ of the twisted curved $\mmS L_\infty$-algebra $\mmg^{-\alpha}$ satisfies $\mu_0^{-\alpha} \in \mmF_{k+1} \mmg^1$. 
		\end{lemma}
		\begin{proof}[Proof of Lemma \ref{Twisting towards the final destination}]
			\label{Proof of Twisting towards the final destination}
			%Aus W50A16 - vgl Kommentar auf W50A21 zur Thematik was der initiale Filtrierungsgrad der Curvature sein muss.
			A direct consequence of the curved $\mmS L_\infty$-algebra equations is the curvature being closed $\mu_1 ( \mu_0) =0$. This particularly means that $\mu_0$ fits into the scheme of Equation \eqref{Spectral Alt} (for $p=k$ and $q=1-k$) and we find
			\begin{align}
				\label{Exists a twister}
				\left\{
				\begin{aligned}
					\exists \alpha \in \mmF_{k-r} \mmg^0~\mst\\
					\mu_1 (\alpha ) \in \mmF_k \mmg^1~\mand~\mu_0 - \mu_1 (\alpha) \in \mmF_{k+1} \mmg^1.
				\end{aligned}
				\right.	
			\end{align}
			Explicitly expanding the expression of the twisted curvature $\mu_0^{-\alpha}$ yields
			\begin{equation}
				\label{Twisted Curvature}
				\mu_0^{-\alpha} \meqc{\eqref{Twisted}} \ubr{\mu_0 - \mu_1 (\alpha)}_{\minc{\eqref{Exists a twister}} \mmF_{k+1} \mmg^1 } + \frac{1}{2!}\hspace{-1em}  \ubr{\mu_2 (\alpha, \alpha)}_{\in \mmF_{2 (k-r)} \mmg^1 \substack{\subset\\k \geq 2r+1} \mmF_{k+1} \mmg^1} \hspace{-2em} + \ldots\in \mmF_{k+1} \mmg^1.
			\end{equation}
		\end{proof}
		%hier klar dass $q \geq 2r+1$ die minimal condition ist, sonst wuerde ja die obige Gleichuhng nicht funktionieren.
		Having stated these small technical lemmata, we can now pursue the proof of the main theorem of this paper
		\begin{theorem}[Main Theorem]
			\label{Main Theorem}
			Let $\mmg$ be a curved $L_\infty$-algebra equipped with a descending, bounded above and complete filtration $	\mmg= \mmF_1 \mmg \supset \mmF_2 \mmg \supset \mmF_3 \mmg \supset \ldots$ compatible with the curved $L_\infty$-algebra structures.\\
			If there exists an integer $r \in \mathbb{N}_0$ for which the curvature $\mu_0$ of $\mmg$ satisfies $\mu_0 \in \mmF_{2r+1} \mmg^2$ and the $r+1$st page of the spec.seq. of $\mmg$ vanishes in total degree 2, i.e. $E_{r+1}^{p,q}=0$ for all $p,q$ with $p+q=2$, then there exists a Maurer-Cartan element $\alpha \in \mMC(\mmg)$ that satisfies $\alpha \in \mmF_{r+1} \mmg$. In particular\footnote{$\alpha \in \mMC(\mmg)$ means $M(\alpha)=0$ and hence $M^R(\alpha \hat{\otimes} R)=0$. $\mmg^\alpha$ being flat on the other hand is tantamount to $M_{\mmg^\alpha}^R(0)=0$, as in $M_{\mmg^\alpha}^R(0)$, due to multilinearity of the higher ($n \geq 1$) brackets, only the curvature $\mu_0^\alpha$ survives. In Equation \eqref{TwistedAlt} we found $M_{\mmg^\alpha}^R (x) = M_\mmg (x + \alpha)$ which eventually shows the statement.}, $\mmg^\alpha$ then forms a non-curved $L_\infty$-algebra.
		\end{theorem}
		\begin{proof}[Proof of the Main Theorem]
			\label{Proof of the Main Theorem}
			As before we do all our calculations in the language curved $\mmS L_\infty$-algebras, i.e. we work with $\mmg[-1]$ instead of $\mmg$  and curved $\mmS L_\infty$-algebra brackets to avoid having ubiquitous suspensions in the expressions.\\
			%Aus W50A16 ff - die AUssage mit Alpha passt schon, siehe W74Z pg. 10
			We describe a procedure of repeated twisting such that in every step the curvature of the twisted $\mmS L_\infty$-algebra gets raised by one. Moreover, we explain why this construction converges and forms a non-curved $\mmS L_\infty$-algebra.\\
			For $k \geq 2r+1$, according to Lemma \ref{Twisting towards the final destination}, there exists an %(we inverted the sign for simplicity)
			$\alpha_1 \in \mmF_{k-r} \mmg^0$ such that the curvature $\mu_0^{\alpha_1}$ of the curved \sLiA $\mmg^{\alpha_1}$ satisfies $\mu_0^{\alpha_1} \in \mmF_{k+1} \mmg^1$. Since $\alpha_1 \in \mmF_{k-r} \mmg^0 \substack{\subset\\ k \geq 2r+1} \mmF_{r+1} \mmg^0$, by virtue of Lemma \ref{Twisting and Spectral} $E_{r+1}^{p,q} (\mmg^{\alpha_1}) =0$ for $p+q=1$ still holds.\\
			As a consequence, we can repeat the very same process with $\mmg$ replaced by $\mmg^{\alpha_1}$ and $k+1$ instead of $k$ and so eventually find an $\alpha_2 \in \mmF_{k-r+1} \mmg^0$ such that the curvature $\mu_0^{{\alpha_1}^{\alpha 2}}$ of $\mmg^{{\alpha_1}^{\alpha_2}}$ satisfies ${\mu_0}^{{\alpha_1}^{\alpha_2}} \in \mmF_{k-r+2}$. Of course this twisted curved $\mmS L_\infty$-algebra still preserves $E_{r+1}^{p,q} (\mmg^{{\alpha_1}^{\alpha 2}}) =0$ for all $p,q$ with $p+q=1$. But (see Equation \eqref{TwistedAlt}) $\mmg^{{\alpha_1}^{\alpha 2}}$ is the same as $\mmg^{\alpha_1 + \alpha_2}$.\\
			In other words we have a way of raising the degree of filtration of the curvature of the twisted curved $\mmS L_\infty$-algebra by one without destroying the condition $E_{r+1}^{p,q}=0$ for $p+q=1$.\\
			Repeated application of this twisting procedure and making use of the completeness of the filtration allows us to eventually obtain a non-trivial twisted curved \sLiA (still on the vector space $\mmg$) with zero curvature. Also notice that by Lemma \ref{Twisting towards the final destination} we have $\alpha_i \in \mmF_{k-r} \mmg^0$, i.e. the element $\alpha_i$ we twist with also increases its degree of filtration in every step. Together with the completeness of the filtration this lets the sum $\alpha_1 + \alpha_2 + \ldots$ converge and therefore makes the construction well-defined.\\
			By means of the equivalence between a curved $\mmS L_\infty$-algebra $\mmg^\alpha$ being flat and $-\alpha$ forming a Maurer-Cartan element of $\mmg$, this shows $\alpha = -\alpha_1-\alpha_2- \ldots$ to satisfy the requirements.
		\end{proof}
		
		\section{Applications}
		\label{Applications}
		We start by presenting the general idea behind using Theorem \ref{MainThm} in the search of $\infty$(-quasi) isomorphisms.\\
		So, let $P$ be a possibly inhomogeneous Koszul operad. A $P_\infty$-algebra, sometimes referred to as strong homotopy algebra\index{strong homotopy algebra}\index{$P_\infty$-algebra}, is an algebra over the cobar-construction $\Omega P^\antishriek$ of the Koszul dual co-operad $P^\antishriek$. Notice, that even though in the definition of inhomogeneous Koszul operads we do not allow for operads with differentials other than the trivial one, the Koszul dual co-operad as well as its cobar construction both do generally carry a non-trivial (co)-differential. Prominent examples of such $P_\infty$-algebras, apart from the $L_\infty$-algebras from before, are $\mathrm{Com}_\infty$-algebras, $\mathrm{A}_\infty$-algebras and $BV_\infty$-algebras.\\
		\begin{remark}
			The main difference in the notion of $P_\infty$-algebras for $P$ being an inhomogeneous Koszul operad, as introduced in \cite{TonksBV}, Appendices A,B, compared to the `classical case' of Koszul operads (with quadratic relations only) is that  a $P_\infty$-algebra still is a $\Omega P^\antishriek$-algebra, but with the construction of $P^\antishriek$ being slightly more involved. More precisely, $P^\antishriek$ is defined as the `classical' Koszul dual of $qP$ with $q$ the map that removes the linear part of from the relations and some special co-differential is given to $(qP)^\antishriek$. For most of the relevant properties, among others the `Rosetta Stone' of equivalent descriptions of $P_\infty$-algebras (see \cite{TonksBV}, Theorem 1), the inhomogeneous Koszul case does not differ from the 'classic' quadratic case. Notice, however, that $\Omega P^\antishriek$, in general, does no longer form a minimal resolution of $P$.
		\end{remark} 
		An $\infty$-morphism between two $P_\infty$-algebras $F: A \rightsquigarrow B$ is defined as a $P^\antishriek$-co-algebra (in dgVect, hence compatible with the respective co-differentials) morphism between the corresponding relative (relative w.r.t. Koszul morphism $\iota: P ^\antishriek \to \Omega P^\antishriek = P_\infty$) bar constructions $F:B_\iota (A) \to B_\iota(B)$ (cf. \cite{TonksBV}, Section B.2).\\
		Recall from \cite{TonksBV}, Section B.1 that the relative bar construction of the $P_\infty$-algebra $A$ is a quasi-co-free co-algebra, hence an $\infty$-morphism $ F: A \rightsquigarrow B$ is completely determined by its projection\footnote{For the projection $\pi_B: \mathbb{F}_{P^\antishriek}^c (B) \to B$ defined by
			\begin{align*}
				\mathbb{F}_{U(P^\antishriek)}^c (B)  \stackrel{\eta \circ Id_B}{\to} I \otimes_{\mathbb{S}_1}B \cong B,
			\end{align*}
			denotes $\epsilon: P^\antishriek \to I$ denotes the co-operadic co-unit.} to $B$, that is a graded vector space morphism $f: \mathbb{F}_{U(P^\antishriek)}^c (U(A)) \to B$.\\
		It is crucial to notice that while every graded vector space morphism $f:\mathbb{F}_{U(P^\antishriek)}^c (U(A))  \to U(B)$ induces a graded $U(P^\antishriek)$-co-algebra morphism $\rho (f):\mathbb{F}_{U(P^\antishriek)}^c (U(A))) \to \mathbb{F}_{U(P^\antishriek)}^c (U(B)) $ (in the underlying category grVect), compatibility with the co-differentials does not hold in general.\\	
		However, there is a way to work on the level of $\mHom_{grVect} (\mathbb{F}_{U(P^\antishriek)}^c(U(A)), U(B))$\footnote{Since we later endow this space with a $\mmS L_\infty$-structure, we do also want to allow for maps of degree other than zero and therefore work with $\mHom_{grVect}$ instead of $\mMor_{grVect}$. Of course, only the degree 0 elements are the ones forming $U(P^\antishriek)$-co-algebra-morphisms.} and analyse which elements induce co-free $U(P^\antishriek)$-co-algebra morphisms that are compatible with the respective co-differentials and as such form $\infty$-morphisms. To this end we define a formal vector field $Q$ on $\mMor_{U(P^\antishriek)-coalg.}(\mathbb{F}_{U(P^\antishriek)}^c (U(A)), \mathbb{F}_{U(P^\antishriek)}^c (U(B)))$ by the commutator with the co-differentials emerging from the bar-construction (that means maps compatible with the differential are the roots of $Q$) and pull this back along $\rho$. This results in a $\mmS L_\infty$-algebra structure on the graded vector space $\mHom_{grVect}(\mathbb{F}_{U(P^\antishriek)}^c (U(A)),U(B))$ and we may call this $\mmS L_\infty$-algebra deformation complex\index{Deformation Complex} $\mathrm{Def}(A,B)$.\\
		Maurer-Cartan elements on the Deformation Complex correspond to graded vector space morphisms whose induced $U(P^\antishriek)$-co-algebra morphisms commute with the differentials, that is to say they induce $\infty$-morphisms. Notice that the zero map also induces a $U(P^\antishriek)$-co-algebra morphism that is compatible with the differentials.\\	
		Let us explain why the search of a Maurer-Cartan element of a curved $\mmS L_\infty$-algebra is of relevance in the construction of $\infty$-(quasi-)isomorphism between $P_\infty$-algebras. For the scope of this motivation let us assume $P$ to be a binary generated (a requirement we will relax in the proper form of the theorem) possibly inhomogeneous  Koszul operad and let $A,B$ be two $P_\infty$-algebras. Consider the situation in which a (quasi-)isomorphism of dg vector spaces ${f_1 : A \to B}$ is given. We may ask: Is there an $\infty$-morphism $F: A \rightsquigarrow B$ which composed with the projection $\pi_B: \mathbb{F}_{P^\antishriek}^c (B) \to B$ yields $f_1$ (and as such by Definition \ref{Infinity (Quasi-)Isomorphism}, $F$ forms an $\infty$-(quasi-)isomorphism). We address this question by inducing a descending, bounded above and complete filtration compatible with the $\mmS L_\infty$-algebra structure on the deformation complex $\mDef(A,B)$ such that elements of higher filtration degree have a vanishing $U(P^\antishriek)(1) \otimes U(A) \to U(B)$ part. We then twist $\mDef(A,B)$ by the graded vector space morphism $\tilde{f}_1 : \mathbb{F}_{U(P^\antishriek)}^c (U(A)) \to U(B)$ that everywhere is zero except on $U(P^\antishriek)(1) \otimes U(A)$, where it is $f_1$. Finding a Maurer-Cartan element $\alpha$ on the twisted deformation complex (for example by cleverly applying Theorem \ref{Main Theorem}) is equivalent to $\tilde{f}_1 + \alpha$ being an $\infty$-morphism. Moreover, if the Maurer-Cartan element $\alpha$ carries a higher degree of filtration it does not alter the $U(P^\antishriek) (1) \otimes U(A)\to U(B)$ part from being $f_1$ and hence yields that $\tilde{f}_1 + \alpha$ is an $\infty$-(quasi-)isomorphism.\\
		When discussing intrinsic formality we will follow this very rationale, start with the identity map as the initial dgVect quasi-isomorphism and search for Maurer-Cartan elements on the twisted deformation complex by e.g. applying Theorem \ref{MainThm}, as these yield the $\infty$-quasi-isomorphism we looking for.\\
		For the remainder of this paper $P$ is assumed to be a possibly inhomogeneous Koszul operad, sometimes (but always mentioned) with restriction on its generating set. Furthermore, we continue using $U$ to denote the forgetful functor $U:dgVect \to grVect$.
		\subsection{Koszul Operads}
		\begin{definition}[Weight Grading on $P^\antishriek$, $\mathbb{F}_{U(P^\antishriek)}^c (U(A))$ and $\mHom_{\mathrm{grVect}} (\mathbb{F}_{U(P^\antishriek)}^c (U(A)),U(B))$]
			%Aus [Tonks] pg. 50
			\label{Weight Grading}
			%\ecomm{NOCH NACHPRUEFEN (Z.B.) TONKS PG. 4 UND 38 OB HIER WIRKLICH IN UNSERER KONVENTION $S$ UND NICH $S^{-1}$}
			Let $P$ be a possibly inhomogeneous Koszul operad. By definition (\cite{TonksBV}, Sect.A.2.) the Koszul dual co-operad is a sub-cooperad of the co-free co-operad $\mathcal{J}^c(sE)$ (with $s$ denoting the suspension operator) for some $\mathbb{S}$-module $E$, which we call generating set. Recalling  from \cite{FreBuch}, Vol.\rnum{1}, Sect.C.1., the definition of the co-free co-operad $\mathcal{J}^c(sE)$ in terms of trees with vertices decorated by $sE$ we may introduce a weight grading on $P^\antishriek$ by the number of vertices (cf. \cite{TonksBV}, Section C.1.) and write ${P^\antishriek}^{(k)}$ for the weight $k$-part of $P^\antishriek$ with respect to this weight grading.\\
			Let $A$ be a $P_\infty$-algebra. We say $v \in P^\antishriek (m) \otimes_{\mathbb{S}_m} A^{\otimes m} =  \prod_{l \geq 0} {P^\antishriek}^{(l)} (m) \otimes_{\mathbb{S}_m} A^{\otimes m}$ has weight $k$ if all the ${P^\antishriek}^{(l)} (m)\otimes_{\mathbb{S}_m} A^{\otimes m}$ are zero except for $l=k$.\\
			Moreover, $w \in \mathbb{F}_{P^\antishriek}^c (A)  = \bigoplus_{n\geq 1} P^\antishriek (n) \otimes_{\mathbb{S}_n} A^{\otimes n}$ is said to be of weight $k$ if all the $v_i \in P^\antishriek (i) \otimes_{\mathbb{S}_i} A^{\otimes i}$ in $w= v_1 + v_2 + \ldots$ are of weight $k$. Both definitions do also hold for their $grVect$ counterparts\\
			Let $A,B$ also be two $P_\infty$-algebras. We introduce a weight grading on \\$\mHom_{grVect} ( \mathbb{F}_{P^\antishriek}^c (U(A)),U(B))$ by setting the weight $k$-part to be the graded vector space homomorphisms that to vanish on all weights of $\mathbb{F}_{U(P^\antishriek)}^c (U(A))$ except possibly for weight $k$.
		\end{definition}
		\subsection{$OpA$-Category}
		%Aus W118A1 bzw. W121A1f.
		Next, let us present the category  $OpA$ consisting of a tuple of an operad $P$ and a $P$-algebra.
		\begin{definition}[$OpA$-category]
			\label{OpA-Cat}
			Let $M$ denote a symmetric monoidal category (e.g. $M=dgVect,grVect$) and let $P$ be an operad in $M$. Furthermore, let $A$ be $P$-algebra.\\
			We establish the category $OpA$ by defining
			\begin{align}
				\label{W118A1}
				Obj(OpA):= (P,A)
			\end{align}
			and setting the morphism to be
			\begin{align}
				\label{W118A2}
				\mathrm{Mor}_{OpA}((P,A),(Q,B))=(f,g),
			\end{align}
			where $f: P \to Q$ is an operad-morphism and $g:A \to f^* B$ is a $P$-algebra morphism.\\
			Recall, that the $P$-algebra $f^* B$ is $B$ with the initial $Q$-algebra structure $\gamma_B : Q \circ B \to B$ made into a $P$-algebra structure $\gamma: P \circ B \to B$ according to
			\begin{align*}
				\gamma: P \circ B \stackrel{f \circ id_B}{\to} Q \circ B \stackrel{\gamma_B}{\to}B.
			\end{align*}
			The composition of morphism works component-wise.
			%vgl auch W121A8f. und W118B6.
		\end{definition}
		\subsection{$P_{(\infty)}$-algebra structures on co-homologies}
		\begin{lemma}[$H(P)$-algebra structure on $H(A)$]
			\label{H(P)-algebra}
			%Aus W120A9 und W121A2
			From \cite{FreBuch}, Vol.\rnum{1}, Prop.3.1.1 it is immediate that every lax monoidal functor also forms a functor of $OpA$ categories.	According to the K\"unneth formula (see \cite{MLHomology}, Chap.8 and \cite{Spanier}, Sect.5.3), the co-homology forms a strong monoidal functor (and hence especially a lax monoidal functor) and so in e.g. $M= dgVect$, the co-homology forms a functor
			\begin{align*}
				H: OpA_{dgVect} &\to OpA_{grVect}\\
				(P,A) &\mapsto (H(P),H(A)).
			\end{align*}
			In particular, the co-homology $H(A)$ of a $P$-algebra $A$ forms a $H(P)$-algebra.
		\end{lemma}
		\begin{definition}[$P_\infty$-algebra structures on $H(A)$]
			% Aus LV Sect 10.3, insbesondere Theorem 10.3.10
			\label{Algebra Strukturen auf auf H(A)}
			Let $P$ be a possibly inhomogeneous Koszul operad and $A$ be a $P_\infty$-algebra. Recall the notion of an operad being possibly inhomogeneous Koszul (cf. \cite{TonksBV}, Sect.A.3), which among others yields the cobar construction of the Koszul dual co-operad to provide a quasi-free resolution of $P$
			\begin{align}
				\label{W114A2}
				\Omega P^\antishriek \substack{\simeq \\ \mathrm{dgOp~quasi-iso}} \hspace{0.2em} P.
			\end{align}
			%ist ein dgOp-quasi-Isomorphism, i.e. the map indced by the projection forms a quasi-Isomormophism.
			%vgl [LV] pg. 259 und [Tonks] pg. 40 -> nur induced by the projection, da gemaess [Tonks] Thm. 39 noch ein Isomorphism dazwischen ist
			But because our definition of Koszul operads does not allow for operads with non-trivial differentials, Equation \eqref{W114A2} also shows
			\begin{align}
				\label{W114A3}
				H( \Omega P^\antishriek) \substack{\cong \\ \mathrm{dgOp Iso}} \hspace{0.1em} P.
			\end{align}
			From Lemma \ref{H(P)-algebra} applied to $\Omega P^\antishriek$ we know that $H(A)$ is endowed with a $H(\Omega P^\antishriek)$-algebra structure. Subsequently, the composition of the structure map $H(\Omega P^\antishriek) \to \mathrm{End}_{H(A)}$ with the morphism of Equation \eqref{W114A3} yields a $P$-algebra structure on $H(A)$.\\
			Furthermore, using the fact every $P$-algebra $B$ also is a $P_\infty$-algebra simply by composition of the structure map with the projection
			\begin{align*}
				P_\infty = \Omega P^\antishriek \stackrel{\sim}{\to} P \to \mathrm{End}_B,
			\end{align*}
			we can also speak of a $P_\infty$-algebra structure on $H(A)$.\\
			Apart from the just found $P$-algebra (and $P_\infty$-algebra structure, respectively) on $H(A)$, there is yet another construction for a $P_\infty$-algebra structure on $H(A)$.\\
			More precisely, by virtue of the homotopy transfer theorem (see \cite{TonksBV}, Lemma 48), we can endow the space $H(A)$ also with a proper $P_\infty$-algebra structure, which we denote by $H(A)_{\mHTT}$ to avoid confusion with the structures from above (even though the underlying graded vector spaces are the same). Among others, $H(A)_{\mHTT}$ has the special property that it extends the $P$-algebra structure from above, that is to say the corresponding $\mathbb{S}$-module morphism $\rho_{H(A)_\mHTT}: P^\antishriek\to \mathrm{End}_{H(B)_{\mHTT}}$ coincides with $\rho_{H(A)}$ in ${P^\antishriek}^{(1)}$.\\
			Moreover, the homotopy transfer theorem allows us to extend the inclusion and projection, respectively, to $\infty$-quasi-isomorphisms (see Definition \ref{Infinity (Quasi-)Isomorphism}) $\iota: H(A)_{\mHTT} \rightsquigarrow A$ and $p:A \rightsquigarrow H(A)_{\mHTT}$.
		\end{definition}
		\subsection{$\infty$-(quasi-)isomorphisms and formality}
		\begin{definition}[$\infty$-(quasi-)isomorphism]
			\label{Infinity (Quasi-)Isomorphism}
			%vgl. W101C4u. fuer Verweise auf passende Was wo (ge)lesen Dokument (Uebersetzung von infinity-morphism Beschreibung als Hom_{S-Mod.) (P^antishriek, End_B^A) nach Hom_{Vect} (F_P^antishriek (A), B)\\
				An $\infty$-algebra morphism $f: A \rightsquigarrow B$ between two $P_\infty$-algebras $A,B$ for a possibly inhomogeneous Koszul operad $P$ is by definition a co-algebra morphism between the co-free $U(P^\antishriek)$-co-algebras $F:\mathbb{F}_{U(P^\antishriek)}^c (U(A)) \to \mathbb{F}_{U(P^\antishriek)}^c (U(B))$ which in addition is compatible with the respective co-differentials emerging from the relative bar-constructions (see Equations \eqref{Relativ Bar tot}-\eqref{Relative Bar 2}). %Naechster Teil aus [LV] pg. 369
				Due to co-freeness such a map $f$ is fully described by a graded vector space morphism $\tilde{f}:\mathbb{F}_{U(P^\antishriek)}^c (U(A)) \to U(B)$ and hence %vgl. pg. W114A10
				%\comm{HIER NOCH MIT INFORMATIONEN AUS MEETING VOM 1.3.2022 KORRIGIEREN}
				using the internal hom adjunction (the underlying symmetric monoidal category is closed) $\mMor_{\mathbb{S}_n\mathrm{Mod.}} (U(P^\antishriek) (n) \otimes_{\mathbb{S}_n} U(A)^{\otimes n},U(B)) \substack{\cong \\\mathrm{Set}} \mMor(U(P^\antishriek) (n)$, $\mHom_{\mathbb{S}_n\mathrm{Mod.}} (U(A)^{\otimes n},U(B)))$ can equivalently be fully characterised as an $\mathbb{S}$-module morphism $\xi \in \mMor_{\mathbb{S}\mathrm{Mod.}} (U(P^\antishriek), \mathrm{End}_{U(B)}^{U(A)})$, where $\mathrm{End}_{U(B)}^{U(A)} = \{\mHom (U(A)^{\otimes n}, U(B)) \}_{n \in \mathbb{Z}_{\geq 1}}$.\\
				For an $\infty$-morphism  the image of the weight 0-part under $\xi$, that is $\xi (P^{\antishriek^{(0)}}):A \to B$, is a chain map.\\% In the general case there are also arity 1 contributions coming from weight 1 and hence do not play a role in the part relevant for $\infty$-(quasi-) isomorphism. Hence not only the linear part, i.e. where A apperas exactly once, but only part of this is relevant for the $\infty$-(quasi-)isomorphism. Nur falls generating Set keinen Beitrag in Airtaet 1 hat (wie es hier aber wirklich auch der Fall ist, haben ja concentrated in arity 1,2) stimmt weight 0 part mit aritaet 1 part ueberein und man kann sagen NUR DER LINEAR (linear ist ja eine Eigenschaft der Aritaet) IST RELEVANT FUER DIE $\infty$-QUASI-ISO-EIGENSCHAFT.
				So it makes sense for an $\infty$-morphism $f: A \rightsquigarrow B$ to be called $\infty$-isomorphism\index{$\infty$-isomorphism} if $\xi (P^{\antishriek^{(0)}})$ is an isomorphism of dg vector spaces.\\
				Similarly, we say an $\infty$-morphism $f: A \rightsquigarrow B$ is an $\infty$-quasi-isomorphism\index{$\infty$-quasi-isomorphism} if $\xi (P^{\antishriek^{(0)}})$ is a quasi-isomorphism of dg vector spaces, that is, it induces an isomorphism in the respective co-homologies.\\
				$\infty$-quasi-isomorphisms  of $P_\infty$-algebras are of great relevance as they do form the weak equivalences in a model category structure\footnote{To be precise it is merely an almost model category structure as even though it admits finite products and pullback of fibrations, in general the category of $P_\infty$-algebras with $\infty$-morphisms fails to admit finite limits and co-limits (see \cite{ValHomotopy}, Section 4.1 and \cite{Loday}, Sect.B.6.3).} on $P_\infty$-algebras with $\infty$-morphisms (cf. \cite{ValHomotopy}, Sect.4.1).
			\end{definition}
			\subsection{(Intrinsic) Formality}
			\begin{definition}[Formality of $P_\infty$-algebras]
				%u.a. LV pg. 424o.	
				\label{Formality}
				Let $A$ be a $P_\infty$-algebra for some possible inhomogeneous Koszul operad $P$. We say that $A$ is formal\index{Formal} if there exists an $\infty$-quasi-isomorphism 
				\begin{align*}
					H(A) \stackrel{\sim}{\rightsquigarrow} A.
				\end{align*}
			\end{definition}
			Besides the notion of formality, there is also the stronger concept of intrinsic formality\index{Intrinsic Formality}, meaning that a $P_\infty$-algebra is called intrinsically formal if every $P_\infty$-algebra $B$ with co-homology $H(B)$ (in the sense of Definition \ref{Algebra Strukturen auf auf H(A)}) that as a $P$-algebra is isomorphic to $H(A)$, is formal itself.
			%aus W81A5
			\begin{definition}[Intrinsic Formality of $P_\infty$-algebras]
				\label{Def Intrinsic Formality}
				%Aus W98A, siehe auch Was ist Wo Ordner fuer Links auf die in W98A erwaehnten Dinge
				%Vgl. W101C6: H(A) isomorph zu H(B) als P-Alg, dann gilt auch $\infty$-isomorph als $P_\infty$-alg
				Let $A$ be a $P_\infty$-algebra for $P$ a possibly inhomogeneous Koszul operad. We say $A$ is intrinsically formal, if for every $P_\infty$-algebra $B$ there is the implication\footnote{The implication can equivalently (using the fact that isomorphic $P$-algebras are in particular also $\infty$-isomorphic and hence $\infty$-quasi-isomorphic as well) be formulated as %vgl. pg. W100C1 f.
					\begin{align*}
						H(B) ~ \substack{\cong \\ P\mathrm{-alg.-iso.}} ~ H(A) ~\implies ~ B ~\substack{\simeq \\ \infty\mathrm{-quasi-iso.}} ~ H(A).
					\end{align*}
					Moreover, $H(B) \substack{\cong \\ P\mathrm{-alg.-iso.}} ~ H(A) ~\implies ~ B \substack{\simeq \\ \infty\mathrm{-quasi-iso.}} ~A$ is another equivalent definition of intrinsic formality, justifying to speak of intrinsic formality being a property of $A$.}
				\begin{align}
					\label{Intrinsic Formality Implication}
					H(B) ~ \substack{\cong \\ P\mathrm{-alg.-iso.}} ~ H(A) ~\implies ~ B ~\substack{\simeq \\ \infty\mathrm{-quasi-iso.}} ~ H(B).
				\end{align}
			\end{definition}
			Before introducing the deformation complex, let us first further elaborate on the notion of derivations of $P$-algebras (and its dual counterparts, respectively), which we have introduced at the end of Section \ref{preliminaries}.
			\subsection{Derivations of $P$-algebras revisited}
			\label{Revisited}
			%AusW121A3
			We start, by using extension of scalars (cf. \cite{Markl}, Section 3.10), to provide an alternative but equivalent definition of derivations of $P$-algebras as $U(P)\otimes R$-algebra morphisms for some graded ring $R$. Furthermore, we extend this definition to derivations with respect to $U(P)$-algebra morphisms. By describing derivations as $U(P)\otimes R$-algebra morphisms we can directly utilise freeness when analysing derivations of free $U(P)$-algebras. Eventually, we translate the notion of a differential of $U(P)$-algebra and a  derivation w.r.t. a $U(P)$-algebra morphism to the $OpA$ language.		
			%Hauptsaechlich aus W110A (+wenig aus W107D,W107A und W107B). Literatur zu diesem Thema: [Getzler,Jones] pg. 22
			\begin{definition}[Derivation with respect to $U(P)$-algebra morphism, equivalent description of $P$-algebra derivations]
				\label{Derivation wrt Morphism}
				Let $P$ be an operad, $A,B$ two $P$-algebras and $f: U(A) \to U(B)$ a graded $U(P)$-algebra morphism.\\
				We say a linear map $\xi: U(A) \to U(B)$ is a derivation of degree k with respect to $f$\index{Derivation with respect to Co-Algebra-Morphism}, if
				\begin{align}
					(f\otimes Id_R+\xi \otimes \epsilon): U(A) \otimes R \to U(B) \otimes R 
				\end{align}
				%vgl. Markl pg. 169 - es muss scheinbar wirklich C \otimes R- Co-Algebra-Morphism sein
				forms a $U(P) \otimes R$-algebra morphism for $R$ the ring $R:= \mathbb{K}[\epsilon]/\epsilon^2$ with $\epsilon$ a formal variable of degree $-k$.\\
				That is to say, the following diagram commutes
				\begin{equation}
					\label{Relative Derivation 1}
					\begin{gathered}
					%\centering
					\tikzstyle{mynode}= [circle,draw=black!50,]
					\begin{tikzpicture}	
						\node [black] (v1) at (-4,2) {$(U(P)\otimes R) (U(A) \otimes R) $};
						\node [black] (v2) at (3,2) {$(U(P)\otimes R) (U(B)\otimes R) $};
						\node [black] (v3) at (-4,0) {$U(A)\otimes R$};
						\node [black] (v4) at (3,0) {$U(B)\otimes R$};
						\node[black,scale=2] (commutative) at (-0.5,1) {$\circlearrowright$};
						\draw[->] (v1) edge node[midway, above, yshift=0.5em, scale=0.75] {$(U(P) \otimes R) (f \otimes Id_R + \xi \otimes \epsilon)$} (v2)  ;
						\draw [->] (v1) edge node [midway,right] {$\gamma$} (v3);
						\draw [->] (v3) edge node[midway, above] {$f \otimes Id_R + \xi \otimes \epsilon$} (v4);
						\draw [->] (v2) edge node[midway, right] {$\gamma$} (v4);
					\end{tikzpicture}.
					\end{gathered}
					%\caption{$P$-derivation w.r.t. $U(P)$-algebra morphism \rnum{1}}
				%	\label{Relative Derivation 1}
				\end{equation}\hfill\\
				Recall, that $(U(P) \otimes R) (f \otimes Id_R + \xi \otimes \epsilon)$ operates on an element in $(U(P)\otimes R) (U(A) \otimes R)$ by leaving the $U(P) \otimes R$-part untouched but simultaneously acting with $f \otimes Id_R + \xi \otimes \epsilon$ on all the $U(A) \otimes R$ decorations. Due to the definition of $R$, only the terms where one or none $\xi \otimes \epsilon$ operation occurs, do survive. From $f$ being a graded $U(P)$-algebra morphism it is clear, that 
				\begin{equation}
					%\centering
					\label{Relative Derivation 2}
					\begin{gathered}
					\tikzstyle{mynode}= [circle,draw=black!50,]
					\begin{tikzpicture}	
						\node [black] (v1) at (-4,2) {$(U(P)\otimes R) (U(A) \otimes R) $};
						\node [black] (v2) at (3,2) {$(U(P)\otimes R) (U(B)\otimes R) $};
						\node [black] (v3) at (-4,0) {$U(A)\otimes R$};
						\node [black] (v4) at (3,0) {$U(B)\otimes R$};
						\node[black,scale=2] (commutative) at (-0.5,1) {$\circlearrowright$};
						\draw[->] (v1) edge node[midway, above, yshift=0.5em, scale=0.75] {$(U(P) \otimes R) (f \otimes Id_R )$} (v2)  ;
						\draw [->] (v1) edge node [midway,right] {$\gamma$} (v3);
						\draw [->] (v3) edge node[midway, above] {$f \otimes Id_R $} (v4);
						\draw [->] (v2) edge node[midway, right] {$\gamma$} (v4);
					\end{tikzpicture}
					\end{gathered}
%					\caption{$P$-derivation w.r.t. $U(P)$-algebra morphism \rnum{2}}
%					\label{Relative Derivation 2}
				\end{equation}\hfill\\
				commutes and hence commutativity of the diagram of Equation \eqref{Relative Derivation 1} is tantamount to commutativity of
				\begin{equation}
					\label{Relative Derivation 3}
					\begin{gathered}
					%\centering
					\tikzstyle{mynode}= [circle,draw=black!50,]
					\begin{tikzpicture}	
						\node [black] (v1) at (-4,2) {($(U(P)\otimes R) (U(A) \otimes R) $};
						\node [black] (v2) at (3,2) {$(U(P) \otimes R) (U(B)\otimes R) $};
						\node [black] (v3) at (-4,0) {$U(A)\otimes R$};
						\node [black] (v4) at (3,0) {$U(B)\otimes R$};
						\node[black,scale=2] (commutative) at (-0.5,1) {$\circlearrowright$};
						\draw[->] (v1) edge node[midway, above,yshift=0.2em, scale=0.75] {$\sum_j Id_{U(P)} \otimes ((f \otimes Id_R) \otimes \ldots \otimes (f \otimes Id_R) \otimes \hspace{-0.5em}  \ubr{(\xi \otimes \epsilon)}_{\text{j th position}} \hspace{-0,5em} \otimes (f \otimes Id_R) \otimes\ldots  \otimes (f \otimes \epsilon))$} (v2)  ;
						\draw [->] (v1) edge node [midway,right] {$\gamma$} (v3);
						\draw [->] (v3) edge node[midway, above] {$ \xi \otimes \epsilon $} (v4);
						\draw [->] (v2) edge node[midway, right] {$\gamma$} (v4);
					\end{tikzpicture}.
					\end{gathered}
%					\caption{$P$-derivation w.r.t. $U(P)$-algebra morphism \rnum{3}}
%					\label{Relative Derivation 3}
				\end{equation}
				%Vgl. pg W119D1 und W119B1AA: Man kann diese aequivalenten Definition nicht Funktoriell einfueren und Aequivalenz dann mittels Adjunction zeigen. Stattdessen muss man extension of scalar machen (vgl. pg W118A8, insbesondere auch dortigen Markl Verweis) und dann als Module mit Coeff in Potenzreihen (Monome) betrachten. Folgedessen muss dann kommutieren der Diagramm fuer die Variabeln einzeln gelten.
			\end{definition}
			When comparing the diagrams of Equation \eqref{Relative Derivation 3} and Equation \eqref{Derivation on A}, one realises that the latter is just the special case of $A=B$ and $f$ being the identity map.\\
			Hence, it makes sense to consider the definition of a derivation with respect to an algebra morphism as a generalisation of the `classical' definition of $P$-derivations from Section \ref{preliminaries} and under this new notion `classical' $P$-derivations correspond to derivations w.r.t. the identity morphism.\\
			One of the main advantages of this generalised definition is that $\mHom_{\mathrm{grVect}} (A,\mathbb{F}_{U(P)} (A)) \cong \mathrm{Der}(\mathbb{F}_{U(P)}(A))$ follows directly from freeness applied to the corresponding $U(P)\otimes R$-algebra morphism.
			\begin{lemma}[Derivations w.r.t. $U(P)$-algebra morphisms and differentials in the $OpA$ language]
				\label{Derivations in OpA Language}
				%Aus W121A4 und W118A2
				Let $P$ be an operad and let $A$ be a $P$-algebra. Consider the space
				\begin{align}
					\label{W118A3}
					\begin{aligned}
						\mathrm{Der}(P,A) := \{\xi=(\xi_1,\xi_2): \xi_1 \in \mHom_{grVect}(P,P), \xi_2 \in \mHom_{grVect} (A,A),\\
						Id+\epsilon \xi \in \mathrm{Mor}_{OpA} ((P,A) \otimes R, (P,A) \otimes R) \},
					\end{aligned}
				\end{align}
				%vgl. Kommentare auf W121A4 warum \epsilon kein OpA-Morphism ist, Id + \epsilon \xi aber einer sein muss
				where $R = \mathbb{K}[\epsilon]/\epsilon^2$ with $\epsilon$ a formal variable of degree 0.\\
				We have the following equivalent characterisation of $P$-derivation and differential of a $P$-algebra $A$ in terms of elements of $\mathrm{Der}(P,A)$.
				\begin{itemize}
					\item A graded linear map $\beta: A \to A$ describes a $P$-derivation of $A$ if
					\begin{align}
						\label{W118A4}
						(0, \beta) \in \mathrm{Der}(U(P),U(A)).
					\end{align}
					\item A graded linear map $\beta: A \to A$ of degree 1, subject to $\beta ^2 =0$ forms a differential on $U(A)$ if
					\begin{align}
						\label{W118A5}
						(d_P, \beta) \in \mathrm{Der}(U(P),U(A))^{(1)}
					\end{align}
				\end{itemize}
				Furthermore, also derivations w.r.t. a $U(P)$-algebra morphism can be characterised in this manner. More precisely, let $f: A \to B$ be a $U(P)$-algebra morphism. Then, a graded linear map $\xi: A \to B$ forms a derivation w.r.t. $f$, if
				\begin{align}
					\label{W118Z1}
					(Id_P,f)+ \epsilon (0,\xi) \in \mathrm{Mor}_{OpA} ((U(P),U(A)) \otimes R, (U(P),U(B)) \otimes R)
				\end{align}
			\end{lemma}
			\begin{proof}[Proof of Lemma \ref{Derivations in OpA Language}]
				For $(Id_P, f)+ \epsilon (0, \xi)$ to form an $OpA$ morphism the following diagram has to commute %vgl. auch Markl pg. 47
				\begin{equation}
					\label{$P$-derivation w.r.t. $U(P)$-algebra morphism in $OpA$-language}
					\begin{gathered}
					%\centering
					\tikzstyle{mynode}= [circle,draw=black!50,]
					\begin{tikzpicture}	
						\node [black] (v1) at (-3,2) {($(U(P)\otimes R) (U(A) \otimes R) $};
						\node [black] (v2) at (3,2) {$(U(P) \otimes R) (Id_{U(P)} \otimes Id_R)^* (U(B)\otimes R) $};
						\node [black] (v3) at (-3,0) {$U(A)\otimes R$};
						\node [black] (v4) at (3,0) {$U(B)\otimes R$};
						\node[black,scale=2] (commutative) at (-0.25,1) {$\circlearrowright$};
						\draw[->] (v1) edge node[midway, above,yshift=0.5em, scale=0.75] {$(Id_{U(P)} \otimes Id_R)\circ (f+ \epsilon \xi)$} (v2)  ;
						\draw [->] (v1) edge node [midway,right] {$\gamma_{A}$} (v3);
						\draw [->] (v3) edge node[midway, above] {$ f + \epsilon \xi$} (v4);
						\draw [->] (v2) edge node[midway, right] {$\gamma_{B}$} (v4);
					\end{tikzpicture}.
					\end{gathered}
%					\caption{$P$-derivation w.r.t. $U(P)$-algebra morphism in $OpA$-language}
%					\label{$P$-derivation w.r.t. $U(P)$-algebra morphism in $OpA$-language}
				\end{equation}
				But this is the very same diagram as in Equation \eqref{Relative Derivation 1}, proving equivalency of the two definitions. Moreover, we notice that Equation \eqref{W118A4} is just the special case $f=Id_A$ of Equation \eqref{W118Z1}. But, as discussed in Definition \ref{Derivation wrt Morphism}, a derivation w.r.t. identity morphism is nothing else than a `classical $P$-derivation'.\\
				It remains to show that Equation \eqref{W118A5} coincides with the definition of differential of a $U(P)$-algebra from Definition \ref{Differential on U(P)}.\\%Naechsters Satz ist schon korrekt, das Id ist ja in der Definition von Derivation (und muss nicht noch separat hinzugefuegt werden
				$(d_P, \beta) \in \mathrm{Der}(U(P),U(A))^{(1)}$ means we have the following commutative diagram
				\begin{equation}
					\label{W118Z2}
					\begin{gathered}
					%\centering
					\tikzstyle{mynode}= [circle,draw=black!50,]
					\begin{tikzpicture}	
						\node [black] (v1) at (-3.5,2) [scale=0.8]{($(U(P)\otimes R) (U(A) \otimes R) $};
						\node [black] (v2) at (3,2) [scale=0.8]{$(U(P) \otimes R) (Id_{U(P)\otimes R}+ \epsilon d_P \otimes Id_R)^* (U(A)\otimes R) $};
						\node [black] (v3) at (-3.5,0) {$U(A)\otimes R$};
						\node [black] (v4) at (3,0) {$U(A)\otimes R$};
						\node[black,scale=2] (commutative) at (-0.25,1) {$\circlearrowright$};
						\draw[->] (v1) edge node[midway, above,yshift=0.5em, scale=0.75] {$(Id_{U(P)} \otimes Id_R)\circ (Id+ \epsilon \beta)$} (v2)  ;
						\draw [->] (v1) edge node [midway,right] {$\gamma_{A }$} (v3);
						\draw [->] (v3) edge node[midway, above] {$Id_A + \epsilon \beta$} (v4);
						\draw [->] (v2) edge node[midway, right] {$\gamma_{A} $} (v4);
					\end{tikzpicture}.
%					\caption{Differential on $U(P)$-algebra in $OpA$-language}
%					\label{W118Z2}
					\end{gathered}
				\end{equation}
				Commutativity of the non-$\epsilon$-part is equivalent to the identity map forming a $U(P)$-algebra morphism and hence is trivially satisfied.\\
				When examining commutativity of the maps containing $\epsilon$, we can restrict to those with exactly one $\epsilon$ because of $\epsilon^2=0$, i.e. we find
				\begin{equation}
					\label{W118Z3}
					\begin{gathered}
					%\centering
					\tikzstyle{mynode}= [circle,draw=black!50,]
					\begin{tikzpicture}	
						\node [black] (v1) at (-4,2) {($(U(P)\otimes R) (U(A) \otimes R) $};
						\node [black] (v2) at (3,2) {$(U(P) \otimes R)  (U(A)\otimes R) $};
						\node [black] (v3) at (-4,0) {$U(A)\otimes R$};
						\node [black] (v4) at (3,0) {$U(A)\otimes R$};
						\node[black,scale=2] (commutative) at (-0.5,1) {$\circlearrowright$};
						\draw[->] (v1) edge node[midway, above,yshift=0.2em, scale=0.75] {$\epsilon d_P \circ Id_{U(A) \otimes R} + Id_{U(P)\otimes R} \circ^\prime (\epsilon \beta)$} (v2)  ;
						\draw [->] (v1) edge node [midway,right] {$\gamma_{A }$} (v3);
						\draw [->] (v3) edge node[midway, above] {$\epsilon \beta$} (v4);
						\draw [->] (v2) edge node[midway, right] {$\gamma_{A} $} (v4);
					\end{tikzpicture},
%					\caption{$\epsilon$ terms of Differential on $U(P)$-algebra in $OpA$-language}
%					\label{W118Z3}
					\end{gathered}
				\end{equation}
				where the $(\epsilon d_P \circ Id_{A \otimes R})$ term is due to $(d_{P \otimes R} + \epsilon d_P \otimes Id_R)^*$.\\
				But this coincides with the diagram from Equation \eqref{Differential on A} and as such implies equivalency of Equation \eqref{W118A5} and Definition \ref{Differential on U(P)}.
			\end{proof}
			All the claims of Section \ref{Revisited} \emph{mutatis mutandis} also hold for their dual counterparts.\\
			\begin{lemma}[Commutator with co-differential]
				\label{Commutator with Differentials}
				%Aus W110A
				%Hier ist ambient Category klarerweise dgVect
				Let $C$ be a co-operad with co-operadic co-differential $d_C$, $A,B$ two $C$-co-algebras with co-differentials $d_A$ and $d_B$, respectively. Moreover, let $f: U(A) \to U(B)$ be a $U(C)$-co-algebra morphism (i.e. compatible with the co-algebraic co-multiplication, but not necessarily with the respective co-differentials).\\
				Then $d_B f - f d_A$ is a co-derivation with respect to $f$.
			\end{lemma}
			
			\begin{proof}[Proof of Lemma \ref{Commutator with Differentials}]
				%Aus W121A7ff., sowie W118A17 und W119D2
				Let $R$ denote the Ring $R = \mathbb{K}[\epsilon]/\epsilon^2$ for $\epsilon$ a formal variable of degree 0. By assumption, $d_A$ is a differential of a $U(C)$-co-algebra and as such by means of Equation \eqref{W118A5} $(d_C,d_A) \in \mathrm{coDer}(U(C),U(A))^{(1)}$, which in turn yields $(Id_C+\epsilon d_C, Id_A+ \epsilon d_A) \in \mMor_{Op^c A^c}((U(C),U(A))\otimes R, (U(C),U(A)) \otimes R)$. Analogously, $(Id_C-\epsilon d_C, Id_B - \epsilon d_B)$ forms a $Op^c A^c$ morphism as well. %Vgl. pg. W121A8m., das passt schon, gemaess v9 Def 2.6. Alle Terme/Abb. kommen auch in obererem horizontalen Pfeil von Figure 1 wegen \circ^\prime genau 1 mal vor. Folgdessen heben sich Vorzeichen gerade auf. Das bedeutet auch, dass falls d_A ein co-Differential ist, dann ist auch -d_A ein co-differential einer $U(C)$-co-algebra. (fuer neue Co-Operadendifferential -d_C). Vgl. [LV] pg. 205 + ibidem pg. 198 auch in Co-Operaden co-differential Gleichung kommt auf beiden Seiten (dort vertifcal) jeweils nur genau 1 d_C-Term gleichzeitig vor, selbes Argument erklaert also auch, warum fuer d_C auch -d_C ein co-Operaden co-Differential beschreibt.
				We further notice that because $f$ is a $U(C)$-co-algebra morphism, also $(Id_C,f)$ forms a $Op^c A^c$ morphism.\\
				%Die obige Zeile folgt direkt aus der Definition von OpA-Kategorie
				Consequently, as a composition of $Op^c A^c$ morphisms
				\begin{align}
					\label{W121Z1}
					(Id_c+\epsilon d_C, Id_A + \epsilon d_A) \circ (Id_C,f) \circ (Id_C - \epsilon d_C, Id_B - \epsilon d_B) 
				\end{align}
				describes a $Op^c A^c$ morphism, too.\\
				However, a short calculation yields
				\begin{align}
					\label{W121Z2}
					\begin{aligned}
						&(Id_c+\epsilon d_C, Id_A + \epsilon d_A) \circ (Id_C,f) \circ (Id_C - \epsilon d_C, Id_B - \epsilon d_B) \\
						&= \big((Id_C + \epsilon d_C) \circ (Id_C) \circ (Id_C - \epsilon d_C), (Id_A + \epsilon d_A) \circ (f) \circ (Id_B - \epsilon d_B)\big)\\
						&=\big(Id_C, f+ \epsilon (d_A \circ f - f \circ d_B)\big)\\
						&=(Id_C,f)+ \epsilon (0, d_A f - f d_B).
					\end{aligned}
				\end{align}
				Recalling Equation \eqref{W118Z1} we know that in order for $(d_A f- f d_B)$ to be a co-derivation w.r.t. $f$ the expression $(Id_C, f) + \epsilon (0,d_A f - f d_B)$ has to form a $Op^c A^c$ morphism. Though, as a composition of $Op^c A^c$ morphism, this certainly is the case.
			\end{proof}
			\subsection{Deformation Complex}
			Deformation complex is an often used expression to denote an $L_\infty$-algebra whose Maurer-Cartan elements form some structures or denote morphisms of a certain kind (e.g. see \cite{Loday}, Sect.12.2 and \cite{What}, Sect.2.1. where the authors call it convolution $\mmS L_\infty$-algebra).\\
			For the scope of this paper we follow \cite{What} and use the notion deformation complex to describe the $\mmS L_\infty$-algebra whose Maurer-Cartan elements are $\infty$-morphisms (see Proposition \ref{MC Elements of Def Complex}).
			\begin{definition}[Deformation Complex]
				\label{Deformation Complex}
				%Aus [Dolg, What] pg. 9
				%Siehe W70B
				%Siehe W102A und W102B
				%Anmerkung zur Frage von Grad von $\infty$-morphism beschreibenden morphism f: $\mathbb{F}_{P^\antishriek}^c (A) \to B$ (aus W114A6): Vgl. [Tonks] und unsere Def von Def-Complex: Ein $\infty$-morphism ist ein dg $P^\antishriek$ co-algebra morphism. Insbesondere ist die underlying category dg Vect und somit haben solche Abb. Grad 0.. Dies ist auch kompatibel mit der Tatsache, dass Def-Complex ist eine $\mmS L_\infty$-algebra und MC-Elemente haben darin haben ebenfalls Grad 0. $U(P^\antishriek)$-coalgebras leben ja in der Kategorie dgVect. Damit ist auch in diesem Fall Morphism von Grad 0.
				Let $P$ be a possibly inhomogeneous Koszul operad and $A,B$ be two $P_\infty$-algebra.\\
				Then we may endow the graded vector space $\mHom_{\mathrm{grVect}}(\mathbb{F}_{U(P^\antishriek)}^{c} (U(A)), U(B))$ with a $\mmS L_\infty$-algebra structure\footnote{See \cite{What}, Appendix A for a proof why this forms a $\mmS L_\infty$-algebra, indeed.}.
				\begin{align}
					\label{DefComplex 1-bracket}
					\{f\}_1 (v) := d_B f(v) - (-1)^{\abs{f}} f (d_{B_{\iota}(A) } v) + \psi_B ( (Id_{U(P)^\antishriek} \otimes f)(\Delta_1 (v)))\\
					\label{DefComplex m-bracket}
					\{f_1,, \ldots, f_m\}_m (v) = \psi_B ((\mathrm{Id}_{U(P)^\antishriek} \otimes f_1 \otimes \ldots \otimes f_m) ( \Delta_m (v)))~ \text{for}~m \geq 2,
				\end{align}
				where $\Delta_m$ is the $m$th component of the co-free co-algebra co-multiplication $\Delta_m: \mathbb{F}_{U(P^\antishriek)}^c (U(A)) \to U(P^\antishriek) (m) \otimes_{\mathbb{S}_m} (\mathbb{F}_{U(P^\antishriek)} ^c(U(A)))^{\otimes m}$ (as described in \cite{Markl}, Equation (3.64)), $d_B$ the internal differential of $B$ (seen as dgVect) and $\psi_B$ is the $P_\infty$-algebra structure describing\footnote{By means of \cite{GetzlerJones}, Prop.2.15, a $P_\infty$-algebra structure on a dg vector space $B$ may equivalently be described by a co-differential-capable $P^\antishriek$-co-derivation on $\mathbb{F}_{P^\antishriek}^c (B)$ and hence can be interpreted (see Definition \ref{Derivation wrt Morphism}) as a $U(P^\antishriek) \otimes R$-co-algebra morphism $(U(P^\antishriek)  \otimes R) ( U(B) \otimes R) \to (U(P^\antishriek) \otimes R) (U(B) \otimes R$) for $R=\mathbb{K}[\epsilon]/\epsilon^2$, where $\epsilon$ is a formal variable in degree $-1$. But because of co-freeness, this may also be described by a graded vector space homomorphism $\psi_B: \mathbb{F}_{U(P^\antishriek)}^c (U(B)) \to U(B)$.} graded vector space homomorphism $\psi_B: \mathbb{F}_{U(P^\antishriek)} ^c(U(B)) \to U(B)$.\\
				Moreover, $d_{B_{\iota}(A)}$ denotes the differential on $\mathbb{F}_{U(P^\antishriek)}^c (U(A))$ emerging from the relative bar construction of $A$ (see Equation \eqref{Relativ Bar tot}). %Notice that we used $\pi_A d_{B_{\iota}(A)} = \psi_A + d_A \epsilon(A)$ for $\epsilon$ being the co-operadic co-unit and $\psi_A$ is the $P_\infty$-algebra structure describing grVect morphism.%Aus LV Sect. 11.2.2 einfach die Terme die Co-Differential induzieren - muessen ja gerade das Resultat der Verkettung mit Projektion sein\\
				We denote this $\mmS L_\infty$-algebra by $\mathrm{Def}(A,B)$ and call it deformation complex\index{Deformation Complex}.
			\end{definition}
			Next, we show that Maurer-Cartan elements of the deformation complex are in 1-to-1 correspondence with $\infty$-morphisms.
			\begin{proposition}[$\infty$-morphisms are Maurer-Cartan elements of the deformation complex]
				\label{MC Elements of Def Complex}
				Let $P$ be a possibly inhomogeneous Koszul operad and let $A,B$ be two $P_\infty$-algebras. Moreover, we assume the deformation complex $\mDef (A,B)$ to be endowed with a descending, bounded above and complete filtration compatible with the $\mmS L_\infty$-algebra structure\footnote{The requirement of such a filtration is mainly to ensure convergence of the Maurer-Cartan equation.}.\\
				Then the Maurer-Cartan elements of the deformation complex $\mDef(A,B)$ correspond to $\infty$-morphisms between the two $P_\infty$-algebras $A$ and $B$.
			\end{proposition}
			\begin{proof}[Proof of Proposition \ref{MC Elements of Def Complex}]
				First, we recall that $\infty$-morphisms $A \rightsquigarrow B$ are defined as vector space morphisms $\mathbb{F}_{U(P^\antishriek)}^c (U(A)) \to \mathbb{F}_{U(P^\antishriek)}^c (U(B))$ that are compatible with the $U(P^\antishriek)$-co-algebra co-multiplications and commute with the respective co-differentials emerging from the relative bar-constructions.\\
				For a given map $f \in \mHom_{grVect}(\mathbb{F}_{U(P^\antishriek)}^c (U(A)),U(B) )$, due to co-freeness, $F:= U(P^\antishriek)(f) \Delta$ describes a $U(P^\antishriek)$-co-algebra morphism, that is a linear map\\ $F: \mathbb{F}_{U(P^\antishriek)}^c (U(A)) \to \mathbb{F}_{U(P^\antishriek)}^c(U(B))$ compatible with the $U(P^\antishriek)$-co-algebra co-multiplication. Hence it comes down to proving that $F$ does also commute with the co-differentials $d_{B_\iota (A)}$, $d_{B_\iota (B)}$ on $\mathbb{F}_{U(P^\antishriek)}^c(U(A))$ and $\mathbb{F}_{U(P^\antishriek)}^c(U(B))$ provided $f$ is a Maurer-Cartan element.\\
				From Lemma \ref{Commutator with Differentials} %(with $P^\antishriek$ in the role of $C$, $\mathbb{F}_{P^\antishriek}^c(A)$ and $\mathbb{F}_{P^\antishriek}^c(B)$ in the role of $A$ and $B$, $d_{B_\iota (A)}$, $d_{B_\iota (B)}$ as co-differentials and $F$ in lieu of $f$)
				we know that the commutator of $F$ with the respective co-differentials forms a co-derivation w.r.t. $F$. 	% Ja, \mathbb{F}_{U(P^\antishriek) \otimes R}^c (U(B) \otimes R) ist korrekt, da ja wegen $\epsilon$^2 =0 eh nur 1 Verzierung mit \epsilon ueberlebt. Erhalten damit dass \mathbb{F}_{U(P^\antishriek) \otimes R}^c (U(B) \otimes R) = $\mathbb{F}_{U(P^\antishriek) \otimes}^c (U(B)) \otimes R
				As discussed in Definition \ref{Derivation wrt Morphism}, a co-derivation w.r.t. $F$ is a $(U(P^\antishriek)\otimes R)$-co-algebra morphism: $\mathbb{F}_{U(P^\antishriek) \otimes R}^c(U(A)\otimes R) \to \mathbb{F}_{U(P^\antishriek)\otimes R}^c(U(B) \otimes R)$. From the target being co-free such a co-algebra morphism, and consequently also the co-derivation, is completely described by its composition with the projection to $B$. Moreover, this means that $F$ commutes with the respective co-differentials if the projection to $U(B)\otimes R$ of the commutator vanishes.\\
				Therefore, the remaining step is to investigate $\pi_B \big( d_{B_\iota (A)}F- F d_{B_\iota (A)} \big)$. Expanding $F$ in terms of $f$ and explicitly writing out $d_{B_\iota(A)}$ (cf. Equations \eqref{Relativ Bar tot}-\eqref{Relative Bar 2}) shows that $\pi_B \big( d_{B_\iota (A)}F- F d_{B_\iota (A)} \big)=0$ and hence $f$ corresponding to an $\infty$-morphism, is tantamount to $f$ satisfying the Maurer-Cartan equation on the deformation complex.
			\end{proof}
			\begin{lemma}[Weight induced filtration deformation complex $\mDef(A,B)$]
				\label{Filtration on Deformation Complex}
				%Aus W94B18
				Let $P$ be a possibly inhomogeneous Koszul operad generated in arities $1,2$, that is to say the generating set of $P$ (by definition a Koszul operad is a quotient operad of the free-operad $\mathcal{J}(E)$ for some $\mathbb{S}$-module $E$, which we call generating set) is non-zero in arities 1 and 2 only. Further, let $A,B$ be two $P_\infty$-algebras.\\
				Recall Definition \ref{Weight Grading}, where we introduced weight filtrations on $P^\antishriek$, $\mathbb{F}_{U(P^\antishriek)}^c (A)$ and $\mHom_{grVect}(\mathbb{F}_{U(P^\antishriek)}^c (U(A)),U(B))$, respectively.
				We say an element $f \in \mDef(A,B)$ is of filtration degree $p$ if it does not contain non-trivial components in weights strictly smaller than $p-1$, i.e. we set
				\begin{equation}
					\label{Filtration}
					\mmF_p \mDef(A,B):= \{f\in \mDef(A,B) : f\vert_{{U(P^\antishriek)}^{(k)} (n) \otimes_{\mathbb{S}_n} U(A)^{\otimes n}}=0~ \forall k<p-1,\forall n\in \mathbb{N}_0 \}.
				\end{equation}
				This endows $\mDef(A,B)$ with a filtration
				\begin{align*}
					\mDef(A,B) = \mmF_1 \mDef(A,B) \supseteq \mmF_2 \mDef(A,B) \supseteq \ldots
				\end{align*}
				that is descending, bounded above, complete and compatible with the $\mmS L_\infty$-algebra structure on the deformation complex.
			\end{lemma}
			\begin{proof}[Proof of Lemma \ref{Filtration on Deformation Complex}]
				\label{Proof of Filtration on Deformation Complex}
				%Dies ist einer der Gruende warum wir Generating Sets auf bestimmte Aritaeten Einschraenken muessen
				By restricting to generating sets of arities $1, 2$, we can ensure the arity $k$ part of $U(P^\antishriek)$ to carry at least weight $k-1$ (using $P^\antishriek$ is as a sub-co-operad of the co-free co-operad $\mathcal{J}^c (s^{-1}E)$).\\
				Notice, that the co-operadic co-multiplication of the co-free co-operad conserves the total (i.e. the sum amongst all the factors in the tensor product) weight and so does the co-free co-algebra co-multiplication on $ \mathbb{F}_{U(P^\antishriek)}^c (U(A))$, cf. \cite{Markl}, Equation (3.64).\\
				Let $f_1, \ldots, f_i$ be such that $f_s \in \mmF_{p_s} \mDef(A,B)$ for $s=1, \ldots i$. Next, we analyse what weight $v\in \mathbb{F}_{U(P^\antishriek)}^c (U(A))$ at least must carry for $\{f_1, \ldots, f_i\}_i (v)$ not to trivially vanish.\\
				Recall that $v \in \mathbb{F}_{U(P^\antishriek)}^c (U(A))$ in general is a combination of elements of the form $v_{t,k} = U(P^{\antishriek})^{(t)} (k) \otimes_{\mathbb{S}_k} U(A)^{\otimes k}$. When trying to find the minimal weight necessary for $v$, we can equivalently look at what happens for $\{f_1, \ldots, f_i\}_i (v_{t,k})$.\\
				By means of Equation \eqref{DefComplex m-bracket} (using \cite{Markl}, Equation (3.64) for the co-algebraic co-multiplication) first the $U(P^{\antishriek})^{(t)}$- part of $v_{t,k}$ is sent to \\$U(P^{\antishriek}) (i) \otimes_{\mathbb{S}_i} \bigg( \bigoplus \mathrm{Ind}_{\mathbb{S}_{l_1} \times \ldots \times \mathbb{S}_{l_m}}^{\mathbb{S}_i} \big( \bigotimes_{s=1}^{m} U(P^\antishriek) (l_s)  \big)  \bigg)$, where the sum is over all $(l_1, \ldots,l_m)$ with $l_s \geq 1$ and $l_1 + \ldots + l_m = k$. Then, by making use of the associator (monoidal category), the $U(P^\antishriek)(l_s)$ are gathered with the $U(A)^{\otimes l_s}$.\\
				As explained above, because the generating set of $P$ is concentrated in arities $1,2$, as an arity $i$ element $U(P^\antishriek) (i)$ must carry weight $i-1$ or bigger.\\
				Moreover, each of the $U(P^\antishriek) (l_s)$ is required to carry weight $\geq p_s-1$ as otherwise application of $f_s$ on $U(P^\antishriek) (l_s) \otimes_{\mathbb{S}_{l_s}} U(A)^{\otimes l_s}$ would trivially vanish because of $f_s \in \mmF_{p_s} \mDef(A,B)$.\\
				Due to the aforementioned conservation of the total weight by the co-free co-algebra co-multiplication, we can proceed by simply adding the minimal requirements for the individual terms we have just found:
				\begin{align*}
					t \geq (i-1)+\ubr{(p_1 -1) + \ldots + (p_i -1)}_{i\text{-many terms}} = (p_1 + \ldots + p_i ) -1.
				\end{align*}
				In other words $v_{t,k}$ must at least have weight $(p_1 + \ldots + p_i)-1$, otherwise it would trivially vanish under $\{ f_1, \ldots, f_i \}_i$. But according to Equation \eqref{Filtration} this means $\{f_1, \ldots, f_i\} \in \mmF_{(p_1+ \ldots+ p_2)} \mDef(A,B)$, hence compatibility of the filtration with the $\mmS L_\infty$-algebra structure holds. \\
				Eventually, $\mmF_1 \mDef(A,B)= \mDef(A,B)$ follows directly from Equation \eqref{Filtration}, thus the filtration is guaranteed to be bounded above.
			\end{proof}
			Notice that the filtration used in Lemma \ref{Filtration on Deformation Complex} differs from the one of \cite{What}, Sect.2.1.1, in the sense that we focus on the operadic weight rather than the arity (for binary generated operads in fact both definitions coincide). The reason for doing so is that if we have a $P_\infty$-algebra structure on $A$, described by $\rho_{\infty}\in \mMor_{\mathbb{S}}(U(P^\antishriek),\mathrm{End}_{U(A)})$, that does extend a $P$-algebra structure $\rho \in \mMor_{\mathbb{S}} (U(P^\antishriek), \mathrm{End}_{U(A)})$, they do differ in weight $\geq 2$ only. Hence, using the weight for filtration is a natural choice, in particular as by this we can get rid of $\rho_\infty$-exclusive terms when going to higher pages of the spectral sequence.\\			
			Now we have a 'good` filtration at hand, hence both the Maurer-Cartan equation and the twisting-procedure are well-defined.\\
			For the scope of this section we will use the following special notation to denote the twisted deformation complex.
			\begin{definition}[Twisted Deformation Complex]
				\label{Twisted DefComplex}
				%Vgl W94A13
				%Gradwise sollte es funktionieren - Dolg. What pg. 9 spricht ebenfalls von shifted L infty structure und Dolg. Enhancement pg. 7 sagt, dass MC-element in mmS L infty Setting Grad 0 hat und fuer unseren Anwendungsfall (f = id) sprechen wir ganz bestimmt von Grad 0 element
				%siehe u.A. W101C2, [FreI] pg. 39 u. (Achtung hier Co-Algebra) und ggf W100B
				Let $P$ be a possibly inhomogeneous Koszul operad generated in arities 1,2 and let $A,B$ be two $P_\infty$-algebras.\\
				Notice, that a graded vector space homomorphism $f:U(A) \to U(B)$ corresponds to an element in $F \in \mDef(A,B)$, namely the graded vector space homomorphism
				\begin{align}
					\label{Corresponding DefComplex element}
					F: \bigoplus_{n \in \mathbb{N}_0} U(P^\antishriek) (n) \otimes_{\mathbb{S}_n} U(A)^{\otimes n} \twoheadrightarrow U(P^\antishriek) (1) \otimes U(A) \stackrel{\epsilon \circ id_{U(A)}}{\to}  \mathbb{K} \otimes U(A) \cong U(A) \stackrel{f}{\to} U(B),
				\end{align}
				where $\epsilon$ denotes the co-operadic co-unit of $P^\antishriek$.\\
				Hence, for a degree zero graded vector space homomorphism $f: U(A) \to U(B)$ it makes sense to twist the deformation complex $\mDef (A,B)$ by the corresponding $F\in \mDef(A,B)$.\\
				The resulting curved $\mmS L_\infty$-algebra we denote by
				\begin{align}
					\label{Eq Twisted DefComplex}
					\mDef (A \stackrel{f}{\to}B) := \mDef(A,B)^F.
				\end{align}
			\end{definition}
			%Hier ist vector space morphism schon das korrekte Wort, denn es ist ja wirklich ein Grad 0 element
			Notice, that this is not the most general notation of a twisted deformation complex since, according to Definition \ref{Twisted curved SLinfty algebra}, we could twist by any degree zero element $F \in \mDef(A,B)$ rather than those induced by degree zero graded vector space homomorphisms $f: U(A) \to U(B)$. However, we motivate this more specific notation by the fact that in Theorem \ref{Intrinsic Formality} we start with a degree zero graded vector space morphism $f: A \to B$ that is a quasi-isomorphism of dg vector spaces, but whose induced co-free $U(P^\antishriek)$-co-algebra morphism fails to commute with the co-differentials from the relative bar constructions and twist the therefore non-flat twisted deformation complex $\mDef (A \stackrel{f}{\to}B)$ by an element $\alpha \in \mmF_2 \mDef(A\stackrel{f}{\to}B)$ resulting in a flat $\mmS L_\infty$-algebra. As explained in a remark in Theorem \ref{Main Theorem}, this is equivalent to $F+\alpha$ forming a Maurer-Cartan element. From $\alpha$ being of filtration degree 2 the twist by $\alpha$ does not change the behaviour on weight 0. But since the weight 0 part is the one responsible for the property of the $\infty$-morphism being an $\infty$-(quasi)-isomorphism, the final element behaves in this regard according to $f$. By explicitly writing $\mDef (A \stackrel{f}{\to} B)$, we highlight this aspect.\\
			\subsection{A sufficient condition for intrinsic Formality}
			With all these definitions and formalism at hand, we are finally ready to work on giving a sufficient condition for intrinsic formality of $P_\infty$-algebras. To this end, we will proceed with some lemmata that will later be used in the proof of Theorem \ref{Intrinsic Formality}.
			\begin{lemma}[Intrinsic Formality in Maurer-Cartan-language]
				\label{Translate to MC-Language}
				Let $A$ be a $P_\infty$-algebra for some possibly inhomogeneous Koszul operad $P$ that is generated in arities 1 and 2.\\
				Then $A$ is intrinsically formal if\footnote{Here, we used the notion from Definition \ref{Algebra Strukturen auf auf H(A)} to emphasise which $P_{(\infty)}$-algebra structure on the dg vector $H(B)$ we consider. By $\mmF_p$ we denote the filtration introduced in Lemma \ref{Filtration on Deformation Complex}.}
				\begin{align}
					\label{Intrinsic Formality Equivalent Definition}
					\mathrm{MC}(\mmF_2 (\mDef (H(B)_{\mHTT}  \stackrel{\mathrm{id}}{\to} H(B)))) \neq \emptyset,
				\end{align}
				for all $P_\infty$-algebras $B$ satisfying $H(B) \substack{\cong \\ P\mathrm{-alg.-iso.}} H(A)$.
			\end{lemma}
		
			\begin{proof}[Proof of Lemma \ref{Translate to MC-Language}]
				We start by observing that $\alpha \in \mMC (\mDef (H(B)_{\mathrm{HTT}} \stackrel{\mathrm{id}}{\to} H(B)))$ is equivalent to $\alpha+ \mathrm{Id}$ being a Maurer-Cartan element of $\mDef(H(B)_{\mHTT},H(B))$ and by Proposition \ref{MC Elements of Def Complex}, $\alpha+ Id$ forming an $\infty$-morphism $\alpha + \mathrm{Id}: H(B)_{HTT} \rightsquigarrow H(B)$.\\
				Moreover, if $\alpha$ is subject to $\alpha \in \mmF_2 (\mDef (H(B)_{\mHTT}, H(B)))$ (the underlying vector space and the filtration do not change under twisting), $\alpha$ vanishes in weight zero. This leads to the conclusion that only the $\mathrm{Id}$-part is responsible for the weight zero part of the just found $\infty$-morphism.\\
				Since, according to Definition \ref{Infinity (Quasi-)Isomorphism}, the property of being an $\infty$-(quasi-)isomorphism does only involve the weight zero part, which here is $\mathrm{id}: H(B)_{\mHTT} \to H(A)$, $\alpha + \mathrm{Id}$ forms an $\infty$-isomorphism, indeed.\\
				Composing this $\infty$-isomorphism with the $\infty$-quasi-isomorphism (by the homotopy transfer theorem, see \cite{TonksBV}, Theorem 49) $p: B ~\substack{\simeq \\ \infty-\mathrm{quasi-iso.}}~H(B)$ results in an $\infty$-quasi-isomorphism  $B~\substack{\simeq \\ \infty\mathrm{-quasi-iso.}} ~ H(B)$.
			\end{proof}
			
			\begin{lemma}
				\label{dg Vect Iso 1}
				Let $B$ be a $P_\infty$-algebra for a possibly inhomogeneous Koszul operad $P$ that is generated in arities 1 and 2.\\
				Then, there is an isomorphism of graded vector spaces
				\begin{align}
					\label{dgVect-Iso Spect Sequencend Def-Complex 1}
					E_1 (\mDef (H(B)_{\mathrm{HTT}} \stackrel{\mathrm{id}}{\to} H(B))) \substack{\cong \\ \mathrm{grVect}} \mDef (H(B) \stackrel{\mathrm{id}}{\to} H(B)),
				\end{align}
				where $E_1$ denotes the first page of the spectral sequence\footnote{In Lemma \ref{Filtration of Curvature} we prove the curvature $\mu_{0,twisted}$ of the twisted deformation complex on the l.h.s. of Equation \eqref{dgVect-Iso Spect Sequencend Def-Complex 1} to satisfy $\mu_{0,twisted} \in \mmF_3 \mDef (H(B)_{\mathrm{HTT}} \stackrel{\mathrm{id}}{\to} H(B))$, hence we can, as explained in Section \ref{Spectral Sequence}, apply the construction of a spectral sequence on a filtered chain complex up to page 2. Moreover, this ensures the differential on the first page to be a proper differential.}.
			\end{lemma}
			\begin{proof}[Proof of Lemma \ref{dg Vect Iso 1}]
				First, we demonstrate that the 1-bracket $\{. \}_1^\mathrm{Id}$ on the deformation complex raises the degree of filtration by one. To this end let us take $f\in \mmF_p \mDef (H(B)_{\mathrm{HTT}} \stackrel{\mathrm{id}}{\to} H(B))$. We examine which weight $v \in \mathbb{F}_{U(P^\antishriek)}^c (H(B)_{\mHTT})$ must carry to not trivially vanish under $\{f\}_1^{\mathrm{Id}}$.\\
				%Aus W94B und W94A
				We notice that all terms originating from the twist by $\mathrm{Id}$ can be ignored, since these terms do all involve at least one $\mathrm{Id}$ term that satisfies $Id \in \mmF_1 \mDef(H(B)_{\mHTT},H(B))$ and hence, due to compatibility of the filtration with the $\mmS L_\infty$-algebra brackets, guarantees the expressions of the form $\{\mathrm{Id}, \ldots, \mathrm{Id}, f\}_{\geq 2}$ to lie in $\mmF_{p+1} \mDef(H(B)_{\mHTT} \stackrel{id}{\to} H(B))$.\\
				Let $v$ be of weight $k$ and of the form $P^{\antishriek (k)} (m) \otimes_{\mathbb{S}_m} {H(B)_{\mathrm{HTT}}}^{\otimes m}$. We continue by analysing the threshold on $k$ for producing a non-vanishing contribution on each of the three terms of
				\begin{align*}
					\{f\}_1 (v) = d_{H(B)} f(v) - (-1)^{\abs{f}} f (d_{B_{\iota}(H(B)_{HTT}) } v) + \psi_{H(B)} ( Id_{U(P)^\antishriek} \otimes f)(\Delta_1 (v))
				\end{align*}
				individually.\\
				Due to $d_{H(B)}=0$, the term $d_{H(B)} f$ vanishes.\\
				For the third term, $\psi_{H(B)} (1 \otimes f)(\Delta_1 (v))$, we know that\\$\Delta_1 (v) \in \sum_{k_1 + k_2 =k} {(P)^\antishriek}^{(k_1)} \bar{\circ} {(P)^\antishriek}^{(k_2)} \bar{\circ} {H(B)_{\mathrm{HTT}}}^{\otimes m}$. %
				%Urspruengliche Formulierung war (ist auch korrekt aber nicht eher too much und schwaecheres Statement reicht bereits voellig aus)From $H(B)$ carrying a $P$-algebra structure only its $P_\infty$-algebra structure inscribing morphism $\rho_{H(B)} \in \mHom_{\mathbb{S} } ((BV)^\antishriek , \mathrm{End}_{H(B)})$ is concentrated in weight 1 (see \cite{Loday}, Proposition 10.1.4.)\comm{DAS IST NICHT DER PUNKT (WIRD ERST SPAETER RELEVANT IN DG VECT ISO) - DER PUNKT IST DASS WIE UNTEN IM 2. TERM AUCH DER P-ALG STRUKTUR MORPHISM NUR IN GEWICHT $\geq 1$ NICHT VERSCHWINDED, DA $\rho (id)=0$, wie z.b. in [LV] pg. 368 beschrieben\\} and so by the internal hom adjunction is $\psi_{H(B)} \in \mHom_{grVect} (\mathbb{F}_{(BV)^\antishriek}^c (H(B)), H(B))$. As such $\psi_{H(B)} (1 \otimes f)$ applied to $\Delta_1(v)$ does trivially vanish unless $k_1 = 1$
				%Lese auch [LV] pg. 368 unter Prop 10.1.11, dann klar dass diese Aussage so gilt
				But the $P_\infty$-algebra structure describing graded $\mathbb{S}$-module morphism $\rho_{H(B)}: U(P^\antishriek) \to \mathrm{End}_{H(B)}$ vanishes in weight 0 (see \cite{Loday}, Sect. 10.1.8) and so does $\psi_{H(B)} \in \mHom_{gr Vect} (\mathbb{F}_{U(P^\antishriek)}^c (H(B)), H(B))$, due the internal hom adjunction. Therefore, $\psi_{H(B)} (1 \otimes f)$ applied to $\Delta_1(v)$ trivially vanishes unless $k_1 \geq  1$. Furthermore, because of $f \in \mmF_p \mDef (H(B)_{\mathrm{HTT}} \stackrel{\mathrm{id}}{\to} H(B))$ the expression does also vanish for $k_2 < p-1$. As a direct consequence of these two observations $k \geq p$ is a necessity for $\psi_{H(B)} (1 \otimes f)(\Delta_1 (v))$ not being trivial and therefore we have $\psi_{H(B)} (1 \otimes f) \Delta_1 \in \mmF_{p+1} \mDef (H(B)_{\mathrm{HTT}} \stackrel{\mathrm{id}}{\to} H(B))$.\\
				%untere Aussage sollte so m.E. schon stimmen - vgl LV pgl 367
				Eventually, there remains the term $f (d_{B_\iota (H(B)_{\mHTT})} v)$, where (cf. \cite{Loday}, Sect.11.2.2) 
				\begin{align}
					\label{Relativ Bar tot}
					d_{B_\iota (H(B)_{\mHTT})} = d_1 +d_2
				\end{align}
				with
				\begin{align}
					\label{Relative Bar 1}
					d_1 = d_{P^{\antishriek}} \circ Id_{H(B)_{\mHTT}} + \ubr{Id_{P^\antishriek} \circ^\prime d_{H(B)_{\mHTT}}}_{=0}
				\end{align} 
				and $d_2$ is the differential-capable $P^\antishriek$-co-derivation induced by the $P_\infty$-structure describing homomorphism $\psi_{H(B)_{\mHTT}} : \mathbb{F}_{U(P^\antishriek)}^c (H(B)_{\mHTT}) \to H(B)_{\mHTT}$, that is (see \cite{Markl}, Vol.\rnum{2}, Prop.3.83)\\
				\begin{align}
					\label{Relative Bar 2}
					d_2 = \sum_{n \in \mathbb{Z}_\geq 1} \Big( \sum_{j=0}^{n-1} Id_{U(P^\antishriek)} \otimes \big( {\pi_{H(B)_{\mHTT}}}^{\otimes j} \otimes \psi_{H(B)_{\mHTT}} \otimes  {\pi_{H(B)_{\mHTT}}}^{\otimes n-j-1} \big)     \Big) \Delta_n,
				\end{align}
				where $\pi_{{H(B)}_{\mHTT}} : \mathbb{F}_{U(P^\antishriek)}^c (H(B)_{\mHTT}) \to H(B)_{\mHTT}$ denotes the projection described in the introduction of Section \ref{Applications}. Due to $d_{H(B)_{\mHTT}}=0$, the $Id_{U(P^\antishriek)} \circ^\prime d_{H(B)_{\mHTT}}$ term vanishes. Moreover, $d_1$ reduces the weight by one since the co-operadic co-differential of a Koszul dual co-operad lowers the weight by one (see \cite{TonksBV}, Sect.C.1.).\\ %genau genommen Seite 50 letzter Satz ueber Proposition 51
				Further, as $\psi _{H(B)_{\mHTT}}$ is non-vanishing in weights $\geq 1$ only and since the co-algebra co-multiplication $\Delta$ leaves the weight unaltered, it follows that also $d_2$ does lower the weight by one.\\
				Therefore, we find $f(d_{B_\iota}(H(B)_{\mHTT}))\in \mmF_{p+1} \mDef(H(B)_{\mHTT} \stackrel{id}{\to} (H(B)))$ and the 1-bracket raises the degree of filtration by one, indeed.\\			
				A direct consequence is that as graded vector spaces (but not as dg vector spaces) the zeroth and the first page of the spectral sequence of  $(\mDef(H(B)_{\mHTT} \stackrel{id}{\to} H(B)), \{.\}_1)$ are isomorphic.\\
				Furthermore, by definition of the zeroth page of the spectral sequence (see Equation \eqref{rthpage}), an element $g \in E_0^{p,q}(\mDef(H(B)_{\mHTT}\stackrel{id}{\to} H(B)))$ can be represented by $g: \mathbb{F}_{U(P^\antishriek)}^c (H(B)_{\mHTT}) \to H(B)$ that is only non-zero in weight p-1.\\ %vgl. W94A7 
				Notice, that as filtered (with the filtration from Lemma \ref{Filtration on Deformation Complex}) graded vector spaces $\mDef(H(B)_{\mHTT} \stackrel{id}{\to} H(B))$ and $\mDef(H(B)\stackrel{id}{\to} H(B))$ do not differ and the same also holds for the respective zeroth pages of the spec.seq.\\
				Consequently, among others using 
				\begin{align*}
					\prod_{p \geq 0} \mmF_p \mmg / \mmF_{p+1} \mmg \hspace{0.2em} \substack{\cong \\ \mathrm{grVect}} \hspace{0.2em} \mmg,
				\end{align*}
				we find 
				\begin{align}
					\label{Vect Iso}
					\begin{aligned}
						\Phi: \prod_{p \geq 0} E_1^{p,q}(\mDef(H(B)_{\mHTT}\stackrel{id}{\to} H(B))) &\substack{\cong \\ \mathrm{grVect}} \prod_{p \geq 0} E_0^{p,q}(\mDef(H(B)_{\mHTT}\stackrel{id}{\to} H(B)))\\
						&\substack{\cong \\ \mathrm{grVect}} \mDef(H(B) \stackrel{id}{\to} H(B)) : \Phi^{-1},
					\end{aligned}
				\end{align}
				where the graded vector space isomorphism $\Phi$ from left to right is mapping a representative $f \in \mmF_p \mDef(H(B)_{\mHTT} \stackrel{id}{\to} H(B))$ to its part that is non-vanishing in weight $p-2$ only and in the other direction we have that $\Phi^{-1}$ is decomposing $\mDef(H(B) \stackrel{id}{\to} H(B))$ into its weight components (see Definition \ref{Weight Grading}) and then take the weight $p-2$, degree q part and map it to $E_1^{p,q} (\mDef(H(B)_{\mHTT}\stackrel{id}{\to} H(B)))$ by means of the quotient map (since the $\{.\}_1^{\mathrm{Id}}$ does raise the degree of filtration by at least one, this is well-defined).
			\end{proof}
			\begin{lemma}
				\label{dg Vect Iso 2}
				The graded vector space isomorphism of Equation \eqref{Vect Iso} is an isomorphism of filtered dg vector spaces.
				%Vgl. Kommentarexemplar 001: Es ist in der Tat so, dass grVect Iso + dg. Morphism impliziert dgVect Iso:
				%Bew.: Betrachte f : V \to W dg Vect Morphismus (kommutiert mit Diff) und grVect Iso. ZZ. f^-1 kommentiert mit Diff. 
				%Sei A \in W bel. Da f grVect Iso $\implies$ \exist a \in V s.t. f(A)=A. Aber dann d_V f^{-1} (A) = d_V f^{-1} (f(a))= d_V (a).
				%Andererseits f^{-1} (d_W (A) )= f^{-1}(d_W f(a)) = f^{-1} ( f (d_V (a))) = d_V (a) und somit ist d_V f^{-1} (A)  =  f^{-1} (d_W (A) ) gezeigt.
			\end{lemma}
			\begin{proof}[Proof of Lemma \ref{dg Vect Iso 2}]
				In order for $\Phi$ to form an isomorphism of dg vector spaces we need
				\begin{align}
					\label{dgVect Iso}
					\Phi (d_{E_1 (\mDef(H(B)_{\mHTT}\stackrel{id}{\to} H(B)))} (.)) = \{ \Phi (.) \}_{1,\mDef(H(B) \stackrel{id}{\to} H(B))},
				\end{align}
				or equivalently
				\begin{align*}
					%\label{dgVect Iso Inverse}
					d_{E_1 (\mDef(H(B)_{\mHTT}\stackrel{id}{\to} H(B)))} (\Phi^{-1} (.)) = \Phi^{-1} (\{ . \}_{1,\mDef(H(B) \stackrel{id}{\to} H(B))} )
				\end{align*}
				for $\Phi$ and $\Phi^{-1}$ the isomorphisms from Equation \eqref{Vect Iso}.
				Let us emphasise that on the left-hand side of Equation \eqref{dgVect Iso} the differential is on the first page, i.e. it is given by $\{.\}_1^{\mathrm{Id}}$ seen as a map $E_{1}^{p,q} \to E_1^{p+1,q}$, so in particular there is a shift in the filtration degree of the domain.\\
				%Dieser rauskommentierte Teil ist doch absolut sinnlos, ich kann einen beliebigen repraesentant nehmen, wenn nicht so waere das ganze nicht wohldefiniert, ich habe aber bereits gesagt, dass klaererweise wohldefiniert ist (ist auch wirklich der Fall!)
				%						Let $f : \mathbb{F}_{(BV)^\antishriek}^c (H(B)_{\mHTT}) \to H(B)$ be such that it vanishes in all weights expect $p-1$ (it is sensible to choose such a representative - chosing another representative that differs from $f$ by being non-zero also in other (higher) weights would alter the result, as applying $\{.\}_1^{\mathrm{Id}}$ to it does increase the weight by at least one and we look at this intermediate result in the first spectral page, where the quotient map....)
				%						\comm{nochmals nachpruefen, aber m.e. fuehrt $\{.\}_1^{\mathrm{Id}}$ nicht nur zu einer erhoehung des filtrationsgrades um 1, sondern auch um eine erhoehung des gewichtes um mindestens 1}
				Let $f : \mathbb{F}_{U(P^\antishriek)}^c (H(B)_{\mHTT}) \to H(B)$ be a graded vector space homomorphism such that it vanishes in all weights except $p-1$. We now go trough all the terms in $\{f\}_{\mDef(H(B)_{\mHTT} \stackrel{id}{\to} H(B))}$ and investigate which of them survive after we apply the quotient map to it. Moreover, we also analyse if the result differs from its counterpart $\{\Phi(f)\}_{\mDef(H(B) \stackrel{id}{\to} H(B)}$, where there is no quotient map involved.\\			
				We first deal with the terms, emerging from the twisting by $\mathrm{Id}$. In general, this twisting leads to additional terms of the form $\{\mathrm{Id}, \ldots,\mathrm{Id},f\}_{\geq 2}$. Recall that the higher brackets, as defined in Equation \eqref{DefComplex m-bracket}, include the graded vector space homomorphism $\psi_{H(B)}: \mathbb{F}_{U(P^\antishriek)}^c (H(B)) \to H(B)$ that describes the $P_\infty$-algebra structure. However, $H(B)$ merely being an $P$-algebra, $\psi_{H(B)}$ is non-zero in weight 1 only and since the generating set is concentrated in arities 1,2, this implies the bracket of highest arity to be the 2-bracket. Therefore, the only additional contribution due to twisting is $\{\mathrm{Id},f\}_2$ and this is non-vanishing in weight $p$ only and as such the term $\{\mathrm{Id}, f\}_2$ remains unchanged when passing to the quotient space.
				Nonetheless, the very same term $\{\mathrm{Id},f\}_2$ also appears on the right-hand side of Equation \eqref{dgVect Iso} as the only contribution due to the twisting on the right-hand side.\\
				We continue the analysis of the composition of $f$ with the differentials by writing them down explicitly using Equation \eqref{dgVect Iso}.\\
				For $d_{H(B)} f$ there is nothing to, as this vanishes from $d_{H(B)}$ being zero, anyway.\\
				The next term we have to consider is $\psi_{H(B)} (Id_{P^\antishriek} \otimes f) (\Delta_1 (v))$. This term does appear on both sides of Equation \eqref{dgVect Iso}, however we still have to make sure that we do not lose any terms due to passing to the quotient space on the left-hand side. From $H(B)$ merely carrying a $P$-algebra structure, its $P_\infty$-algebra structure characterising map $\rho_{H(B)}: P^\antishriek \to \mathrm{End}_{H(A)}$ and $\psi_{H(B)} \in \mHom_{gr Vect} (\mathbb{F}_{U(P^\antishriek)}^c (H(B)), H(B))$ by the internal hom-adjunction, respectively, are concentrated in weight 1. Hence, $\psi_{H(B)} (1 \otimes f)\Delta_1$ is potentially non-zero only for $v \in \mathbb{F}_{U(P^\antishriek)}^c (H(B))$ being of weight $p$, as this is the only case in which $\Delta_1$ can allocate a weight 1 term in the first and a weight $p-1$ term in the second co-operadic factor. Consequently $\psi_{H(B)} (1 \otimes f) \Delta_1$ cannot be cancelled out by a term in $ \mmF_{p+2} \mDef(H(B)_{\mHTT} \stackrel{id}{\to} H(B))$ as such a term by definition must vanish on all weights less than or equal to $p$.\\
				By the same arguments that showed $\{.\}_1$ to raise the degree of filtration by one it follows that for $f \in \mmF_p \mDef (H(B)_{\mHTT} \stackrel{\mathrm{id}}{\to} H(B))$ being concentrated in weight $p-1$, the expression $f d_1$ is possibly non zero only on weight $p$. Hence, the contribution due to $d_1$ remains unchanged when passing to the quotient space on the left-hand side of Equation \eqref{dgVect Iso}. For the right-hand side, where there is no quotient space and $H(B)$ instead of $H(B)_{\mHTT}$, the exact same terms emerge from $f d_1$.\\% Problematik von Anwenden von HTT auf extendend Koszul Operads: IST HOMOTOPY TRANSFER THEOREM AUCH FUER EXTENDED KOSZUL DEFINITION ANWENDBAR - vgl. [Dummond Cole, Vallette; The minimal model of the BV operad]  am Ende von Thm. 6.2. steht, dass dies nicht so einfach geht, man aber das ganze schon extenden kann - mir ist klar, dass dies fuer BV kein Problem ist, aber fuer allgemeine extended Koszul Operaden (muessen dort dann eh noch darueber diskutieren, wie $P_\infty = \Omega P^\antishriek$ ? definiert ist) -> das ist fuer uns kein Problem, der Punkt ist dass man dort eine minimale Aufloesung haben will, und unsere $BV^\antishriek$ Definition keine Minimale Aufloesung bietet (sogar allgemein fuer $P^\antishriek$ keine minimale Aufloesung). Fuer minimale Aufloesung muesste man mit Homotopy Operads (a.k.a. $\infty$-oeprads) arbeiten (wie in [LV] eingefuert) und dabei wird die operadische co-composition zu einer $\infty$-operadischen co-composition und das ist, was das Problem erzeugt.
				We proceed by analysing $f d_2$. By virtue of the homotopy transfer theorem, the $P_\infty$-algebra structure on $H(B)_{\mHTT}$ extends the $P$-algebra structure on $H(B)$, that is to say on ${P^\antishriek}^{(1)}$ the $P_\infty$-algebra-structure describing morphisms\\$ \phi_{H(B)} \in \mMor_{\mathbb{S}\mathrm{Mod}} (U(P^\antishriek) , \mathrm{End}_{H(B)})$ and  $\phi_{H(B)_{\mHTT}} \in \mHom_{\mathbb{S}\mathrm{Mod}} (U(P^\antishriek) , \mathrm{End}_{H(B)_{\mHTT}})$ coincide. By means of the internal hom-adjunction this also implies $\psi_{H(B)} \in \mHom_{grVect} (\mathbb{F}_{U(P^\antishriek)}^c H(B), H(B))$ and $\psi_{H(B)_{\mHTT}} \in \mHom_{grVect} (\mathbb{F}_{U(P^\antishriek)}^c H(B)_{\mHTT}, H(B)_{\mHTT})$ (where we write $H(B)$ and $H(B)_{\mHTT}$ for clarity even though as dg vector spaces they coincide) to be equal in weight $1$.\\
				After noticing that, according to Equation \eqref{Relative Bar 2}, the term $\psi_{H(B)_{\mHTT}}$ appears in $f d_2$ exactly once, we continue by decomposing $\psi_{H(B)_{\mHTT}}$ in its weight 1 part $\psi_{H(B)_{\mHTT}}^{(1)}$, which is just $\psi_{H(B)}$, and its weight $\geq 2$ part that we will denote by $\psi_{H(B)_{\mHTT}}^{\geq 2}$ and call their $d_2$ contributions $d_2^{(1)}$ and $d_2^{(\geq2)}$, respectively. A short investigation shows that $fd_2^{(1)}$ is non-zero in weight $p$ only and as such does not get lost when passing to the quotient space. On the other hand the appearance of $\psi_{H(B)_{\mHTT}}^{\geq 2}$ in $f d_2^{(\geq2)}$ requires at least an additional weight $2$, making it a total of at least $p+1$. Hence $fd_2^{(\geq2)}$ becomes zero under the quotient map.\\
				But this is in line with the right-hand side of Equation \eqref{dgVect Iso} as there we only have the $P$-algebra structure $\psi_{H(B)}$ on $H(B)$  (which by above reasoning is equal to $\psi_{H(B)_{\mHTT}}^{(1)}$) and consequently validates Equation \eqref{dgVect Iso}.\\
				By carefully going through all the previous steps one can verify that $\Phi$ even is an isomorphism of filtered chain complexes.
			\end{proof}
			\begin{lemma}[Filtration of Curvature]
				\label{Filtration of Curvature}
				Let $B$ be a $P_\infty$-algebra for a possibly inhomogeneous Koszul operad $P$ that is generated in arities 1 and 2. Then, the curvature $\mu_{{0,\mathrm{twisted}}}$ of the twisted deformation complex $\mDef (H(B)_{\mathrm{HTT}} \stackrel{\mathrm{id}}{\to} H(B))$ satisfies 
				\begin{align}
					\label{Curvature Filtration}
					\mu_{{0,\mathrm{twisted}}} \in \mmF_3 \mDef (H(B)_{\mathrm{HTT}} \stackrel{\mathrm{id}}{\to} H(B)).
				\end{align}
			\end{lemma}
			\begin{proof}[Proof of Lemma \ref{Curvature Filtration}]
				By unravelling the definitions of the twisted deformation complex and the filtration thereon, we see that for $\mu_{0,\mathrm{twisted}}$ to satisfy the requirement of Equation \eqref{Curvature Filtration} is tantamount to the graded vector space homomorphism $\mathbb{F}_{U(P^\antishriek)}^c (H(B)_{\mathrm{HTT}}) \to H(B)$, given by $M(\mathrm{Id})$, to vanish in weights $1$ and $2$.\\
				To this end, we notice that $id$ induces an $\infty$-morphism $\tilde{Id}: H(B) \rightsquigarrow H(B)$ for $H(B)$ the $P$-algebra from the first part of Definition \ref{Algebra Strukturen auf auf H(A)} but seen as a $P_\infty$-algebra, and as such constitutes for a Maurer-Cartan element of $\mDef(H(B) \to H(B))$. Consequently the curvature $\tilde{\mu}_{{0,\mathrm{twisted}}}$ (we added an extra $\sim$ to our notation to distinguish it from the curvature of $\mDef (H(B)_{\mathrm{HTT}} \stackrel{\mathrm{id}}{\to} H(B))$) of the twisted deformation complex $\mDef(H(B) \stackrel{id}{\to} H(B))$ is zero.\\
				As previously discussed in Lemmata \ref{dg Vect Iso 1} and \ref{dg Vect Iso 2}, $\mDef(H(B),H(B))$ and \\$\mDef(H(B)_{\mHTT},H(B))$ along with the twisted counterparts, respectively, do only differ by the latter having some additional terms in its brackets due to the occurrence of a $P_\infty$-algebra structure rather than a $P$-algebra structure. This remains valid also for the curvatures $\mu_{0,\mathrm{twisted}}$ and $\tilde{\mu}_{{0,\mathrm{twisted}}}$ of the twisted deformation complexes. That being said, due to $\tilde{\mu}_{{0,\mathrm{twisted}}}$ being zero, we may neglect all the terms of $\mu_{0,\mathrm{twisted}}$  not involving weight $\geq 2$ $P_\infty$-algebra (and hence $H(B)_{\mHTT}$ specific) terms.\\
				The remaining terms are $\mathrm{Id} ~d_2^{(\geq 2)}$, i.e. those from $d_2$ that do involve $\psi_{H(B)_{\mHTT}} : \mathbb{F}_{U(P^\antishriek)}^c (H(B))_{\mHTT} \to H(B)$ in weight $\geq 2$ only and therefore clearly need weight $\geq 2$ to attribute a non-zero contribution with.\\
				This, however, directly validates the statement $\mu_{0,\mathrm{twisted}} \in \mmF_3 \mDef (H(B)_{\mHTT} \stackrel{id}{\to} H(B))$.
			\end{proof}
			\begin{theorem}
				\label{Main Application Theorem}
				Let $A$ be a $P_\infty$-algebra for a possibly inhomogeneous Koszul operad $P$ that is generated in arities 1 and 2.\\
				If the twisted deformation complex $\mDef (H(A) \stackrel{\mathrm{id}}{\to} H(A))$ is acyclic in total degree 1 (i.e. $H^p(\mmF_q \mDef(H(A))\stackrel{\mathrm{id}}{\to} H(A)))=0$ for all p,q with $p+q=1$), then $A$ is intrinsically formal as a $P_\infty$-algebra.
			\end{theorem}
			\begin{proof}[Proof of Theorem \ref{Main Application Theorem}]
				By means of Lemma \ref{Translate to MC-Language} it suffices to show that all $P_\infty$-algebras $B$ with co-homologies isomorphic to $H(A)$ are subject to $\mMC(\mmF_2 (\mDef(H(B)_{\mHTT}  \stackrel{id}{\to} H(B)))) \neq 0$.\\
				%Notice: The reason we have $H(B)_{\mHTT}$ (instead of $H(A)_{\mHTT}$) is that the homotopy transfer theorem has to be applied to $H(B)$ (being seen as a deformation retract of $B$ - therefore it makes a difference whether we look at $H(B)$ or at $H(A)$ even thoughboth are isomorphic as $P$-algebras) to receive $\infty$-quasi-isomorphisms $\iota: H(B) \rightsquigarrow B$ and $p: B \rightsquigarrow H(B)$ (cf. \cite{Loday}, Theroem 10.3.10).
				Because of the co-homologies being isomorphic, also
				%siehe Meeting vom 15.3.2022, sollte doch so funktionieren: Der Isomorphism ist ja H(B) nach H(A) als P-Algebra. um Equation \eqref{W144A15} zu erhalten, muss ich ja einfach auf l.h.s. mit dem induzierten $P_\infty$-algebra isomorphism (und auf r.h.s. mit Isomorphism. selbst) verketten. Der induzierte Isomorphism auf der l.h.s. ist aber gerade $U(P^\antishriek) (Iso)$ und als solches gegeben durch $Id_P^\antishriek \otimes (Iso \otimes \ldots\otimes Iso)$. Aber damit ist klar, dass auf operadischem Teil nur $Id$ wirkt und somit das Gewicht nicht veraendert wird. Aber das bedeuetet, dass das Zwischenschalten dieses induzierten Iso fuer Gewichte keine Relevanz hat. Demzufolge wird die Filtration nicht veraendert
				\begin{align}
					\label{W114A15}
					\mDef(H(B) \stackrel{id}{\to} H(B))  \substack{\cong \\ \text{filtered}~ dg Vect}  \mDef(H(A) \stackrel{id}{\to} H(A))
				\end{align}
				has to hold.\\
				Moreover, we can apply Lemma \ref{dg Vect Iso 2} on $B$ and, after invoking Equation \eqref{W114A15}, find
				\begin{align*}
					E_1 (\mDef(H(B)_{\mHTT} \stackrel{\mathrm{id}}{\to} H(B))) \substack{\cong \\ \text{filtered} ~dg Vect} \mDef (H(A) \stackrel{\mathrm{id}}{\to} H(A)).
				\end{align*}
				Taking co-homologies and using both, the initial assumption of acyclicity $\mDef(H(A) \stackrel{id}{\to} H(A))$ in total degree $1$ and the fact that the second page of a spectral sequence is isomorphic to the co-homology of the first, we obtain
				\begin{align}
					\label{DefComplex vanishes}
					E_2^{p,q} (\mDef (H(B)_{\mathrm{HTT}} \stackrel{\mathrm{id}}{\to} H(B))) =0
				\end{align}	
				for all $p,q$ with $p+q=1$.\\
				Because we also know from Lemma \ref{Filtration of Curvature} that the filtration degree of the curvature of $\mDef(H(B)_{\mathrm{HTT}} \stackrel{id}{\to} H(B))$ is at least 3, this puts us in a situation in which all assumptions of Theorem \ref{Main Theorem} (for $r=1$) are satisfied. Consequently, a Maurer-Cartan element $\alpha \in \mathrm{MC}(\mmF_2 \mDef(H(B)_{\mHTT} \stackrel{id}{\to} H(B)))$ must exist. 
			\end{proof}
			\appendix
			\section{Proof of Equivalence of $(\mmS)L_\infty$-algebra equations}
			\label{Appendix}
			\begin{lemma}[Equivalence of curved $(\mmS)L_\infty$-algebra Equations \rnum{1}]
				\label{Powereq also works}
				Let $\mmg$ be a curved $\mmS L_\infty$-algebra endowed with a filtration that is descending, bounded above, complete and compatible with the $\mmS L_\infty$-algebra structure. Then the curved $\mmS L_\infty$-algebra Equation \eqref{LinftyEq}\footnote{To be precise, Equation \eqref{LinftyEq} describes the curved $L_\infty$-algebra equation. But as explained in Definition \ref{SLinfty Algebra}, curved $\mmS L_\infty$-algebra brackets do share the very same condition but with the brackets being maps $\mu_n: S^n (\mathfrak{g}) \to \mathfrak{g}$ instead of $\mu_n: S^n (\mathfrak{g}[-1]) \to \mathfrak{g}[-1]$.} can equivalently be encoded in the form of Equation \eqref{powereq}.
			\end{lemma}
			%Passt schon mit graded Ring statt graded algebra (vgl. [B.F., T.W., V.T.; Homotopy of ....]pg. 19 ff. vs [B.F.,T.W.; Mapping] pg. 22)
			%%Zwiebach pg. 6- pg. 9; m.E. ist der Punkt auch, dass L_\infinity Algebra Struktur auch als Square Zero Coderivation (S(V) ist ja eine Coalgebra) Q definiert werden kann. Eine Coderivation ist vollstandig durch lineare Abb F:S(V) -> V definiert. Laut [B.F., T.W.,V.T.; The Rational Homotopy of Mapping Spaces of E_n Operads] pg.19f. ist aber eine solche Abbildung vollstaendig durch power series beschrieben.
			\begin{proof}[Proof of Lemma \ref{Powereq also works}]
				%Siehe Anhang von W45 A
				The implication \eqref{powereq} $\Rightarrow$ \eqref{LinftyEq} follows directly from graded polarisation (see Section \ref{power}).\\
				For the other direction we first recall from Equation \eqref{MR} that we can extend the $\mmS L_\infty$-algebra structure to $\mmg \hat{\otimes} R$ for $R$ being a nilpotent graded ring. We may rewrite Equation \eqref{LinftyEq} for some fixed $n$ in the case of all $x_i$ being set to some arbitrary $x \in (\mmg \hat{\otimes} R)^0$ and get
				\begin{align}
					\label{Linfty intermediate}
					\sum_{\substack{j,k >0\\j+k=n}} \binom{j+k}{k}\mu_{j+1} (\mu_k(x, \ldots,x),x, \ldots, x,) =0,
				\end{align}
				since $x$ is of degree 0 and therefore we can reorder without picking up additional signs for all the $\binom{j+k=n}{k}$-many $(j,k)$-unshuffles.\\
				Summing over $n$, adding an additional factor $\frac{1}{n!}$ (for each fixed $n$ the summand vanishes separately) and using $\binom{j+k=n}{k}= \frac{n!}{j!k!}$ leads to\\
				\begin{align*}
					&\sum_{n\geq0} \sum_{\substack{j,k >0\\j+k=n}}\frac{1}{j!k!} \mu_{j+1} (\mu_k (x, \ldots, x),x \ldots,x) \\
					&=  \sum_{j\geq 0} \frac{1}{j!} \mu_{j+1} (\sum_{k\geq0} \frac{1}{k!} \mu_k (x, \ldots, x), x \ldots, x)\\
					&= \sum_{l \geq 1} \frac{1}{(l-1)!} \mu_{l} (M^R(x), x \ldots, x)\\
					&= \sum_{l \geq 1} \frac{l}{l!} \mu_{l} (M^R(x), x \ldots, x)
					\substack{=\\\eqref{DM explicit}} DM^R(x)[M^R(x)] =0.
				\end{align*}
			\end{proof}
			%Aus pg W44A12ff fuer Proof
			\begin{lemma}[Equivalence of curved $(\mmS)L_\infty$-algebra Equations \rnum{2}]
				\label{Both powereqs work}
				Let $\mmg$ be a curved $\mmS L_\infty$-algebra endowed with a filtration that is descending, bounded above, complete and compatible with the $\mmS L_\infty$-algebra structure. Then the curved $\mmS L_\infty$-algebra Equation \eqref{powereq} can equivalently be encoded in the form of Equation \eqref{powereqalt}.
			\end{lemma}
			\begin{proof}[Proof of Lemma \ref{Both powereqs work}]
				%Siehe Viedo vom 19.11.2020 um 12:00 und vgl. auch mit Dolg. GMT pg. 8
				Let $x$ be $x \in (\mmg \hat{\otimes} R)^0$ for $R:= \mathbb{K}[\epsilon_1, \ldots, \epsilon_n]/(\epsilon_1^2, \ldots, \epsilon_n^2)$ where $\epsilon_i$ is a formal variable of degree $-\abs{x_i}$.\\ % i.e. x= x_1 \otimes \epsilon_1 + \ldots + x_n \otimes \epsilon_n$
				Using the definition of the derivative of a multilinear map %vgl W39A1 und W44A13
				combined with the fact that for $\abs{x}=0$ graded symmetry allows us to change the order of arguments freely without picking up additional signs, we may find
				\begin{align}
					\label{DM explicit}
					DM^R(x)[M^R(x)] = \sum_{n \geq 1} \frac{1}{n!} \sum_{j=1}^n \mu_n(x, \ldots, x, \ubr{M^R(x)}_{j~\text{th position}},x , \ldots, x)\\
					\nonumber =\sum_{n \geq 1} \frac{n}{n!}\mu_n (M^R (x),x,\ldots,x) .
				\end{align}
				But on the other hand  %der $M^R(x)$- Term ist derjenige, welcher (Multilinearitaet!) kein $\epsilon$ hat.
				\begin{align*}
					M^{R \otimes R^\prime} (x\hat{\otimes}1+ M^R (x)\hat{\otimes}\epsilon) = \sum_{n=0}^\infty \frac{1}{n!} \mu_n ( x\hat{\otimes}1 + M^R (x)\hat{\otimes} \epsilon, \ldots, x\hat{\otimes}1+ M^R (x)\hat{\otimes} \epsilon) \\
					\meqc{\epsilon^2=0} M^R(x)\hat{\otimes}1 + \sum_{n \geq 1} \frac{1}{n!} \sum_{j=1}^n \mu_n (x\hat{\otimes}1 \ldots,x \hat{\otimes}1, \ubr{M^R(x)\hat{\otimes} \epsilon}_{j~\text{th position}}, x\hat{\otimes}1, \ldots, x\hat{\otimes}1)\\
					= M^R(x)\hat{\otimes}1 + \sum_{n \geq 1} \frac{n}{n!} \mu_n ( M^R (x)\hat{\otimes} \epsilon, x \hat{\otimes}1,\ldots,x\hat{\otimes}1)
				\end{align*}
				also holds, and so it is immediate that $M^{R \otimes R^\prime}(x\hat{\otimes}1+ M(x)\hat{\otimes}1) = M^R(x)\hat{\otimes}1$ is tantamount to $DM^R(x)[M^R(x)]=0$.
			\end{proof}

			%\nocite{*} %includes also all non-used citations
			%\printindex
			\bibliographystyle{plain}
			\bibliography{literaturverzeichnis}
		\end{document}